\documentclass[10pt]{article}
\usepackage[margin=0.67in,top=0.37in,bottom=0.57in]{geometry}
\usepackage{amssymb}
\usepackage{amsmath}
\allowdisplaybreaks
\usepackage{bbm}
\usepackage{enumerate}
\usepackage{float}
\usepackage{rotating}
\usepackage{setspace}

\usepackage{amsthm}

\theoremstyle{definition}

\newcommand{\mat}[4]{\left[ \begin{array}{cc} #1 & #2 \\ #3 & #4 \end{array} \right] }

\newcommand{\RR}{\mathbbmss{R}}
\newcommand{\CC}{\mathbbmss{C}}

\newcommand{\mary}[2]{{#1}^{\parallel #2}}

\DeclareMathOperator{\matt}{\bf{mat}}
\DeclareMathOperator{\ten}{\bf{ten}}
\DeclareMathOperator{\DCT}{\bf{DCT}}
\DeclareMathOperator{\IDCT}{\bf{IDCT}}

\DeclareMathOperator{\rk}{rank}
\DeclareMathOperator{\prob}{Prob}
\DeclareMathOperator{\ind}{Ind}

\DeclareMathOperator{\pinv}{pinv}

\DeclareMathOperator{\svd}{svd}
\DeclareMathOperator{\dct}{dct}
\DeclareMathOperator{\eye}{eye}
\DeclareMathOperator{\diag}{diag}
\DeclareMathOperator{\vecc}{vec}
\DeclareMathOperator{\ones}{ones}
\usepackage[english, algosection, algoruled, noline]{algorithm2e}

\usepackage{eucal}
\newcommand{\ac}{\mathcal{A}}
\newcommand{\bc}{\mathcal{B}}
\newcommand{\cc}{\mathcal{C}}

\newcommand{\gc}{\mathcal{G}}
\newcommand{\ic}{\mathcal{I}}
\newcommand{\pc}{\mathcal{P}}
\newcommand{\qc}{\mathcal{Q}}

\newcommand{\uc}{\mathcal{U}}
\newcommand{\vc}{\mathcal{V}}
\newcommand{\scc}{\mathcal{S}}
\newcommand{\xc}{\mathcal{X}}

\newcommand{\wc}{\mathcal{W}}

\newtheorem{example}{Example}[section]
\newtheorem{theorem}{Theorem}[section]

\newtheorem{lemma}{Lemma}[section]
\newtheorem{definition}{Definition}[section]

\title{The generalized inverses of tensors via the C-product}

\author{
Hongwei Jin,\thanks{ School of Mathematics and Physics;  Center for Applied Mathematics of Guangxi, Guangxi Minzu University, 530006, Nanning, PR China. { E-mail address}:  jhw$\_$math@126.com} \quad
Shumin Xu,\thanks{School of Mathematics and Physics, Guangxi Minzu University, 530006,
Nanning, PR China. { E-mail address}: 18438596256@163.com }\quad
Hongjie Jiang,\thanks{Corresponding author. School of Mathematics and Physics; Center for Applied Mathematics of Guangxi, Guangxi Minzu University, 530006,
Nanning, PR China. { E-mail address}:  hongjiejiang@yeah.net}\quad
Xiaoji Liu\thanks{School of Education, Guangxi Vocational Normal University, 530007, Nanning, PR China. { E-mail address}:  xiaojiliu72@126.com}
}

\date{}

\begin{document}

\begin{spacing}{1.1}

\maketitle

\begin{abstract}

This paper studies the issues about the generalized inverses of tensors under the C-product. The aim of this paper is threefold.  Firstly,  this paper  present the definition of the Moore-Penrose inverse, Drazin inverse of tensors under the C-product. Moreover, the inverse along a tensor is also introduced. Secondly, this paper gives some other expressions of the generalized inverses of tensors by using several decomposition forms of tensors. Finally, the algorithms for computing the Moore-Penrose inverse, Drazin inverse of tensors and the inverse along a tensor are established.
\\


\end{abstract}


\section{Introduction}

In recent years, the studies of tensors or the multidimensional array have become more popular. A complex tensor can be regarded as a multidimensional array of data, which takes the form $\ac=(a_{{i_1}\ldots{i_p}})\in\CC^{n_1\times n_2\times \cdots\times n_p}$. The order of a tensor is the number of dimensions which is also called ways or modes. Therefore, the well-known vectors and matrices are called first-order tensors and second-order tensors. This paper studies the third-order tensors.

Higher-order tensors have been used in various fields, such as psychometrics\cite{K}, chemometrics\cite{NCT}, face recognition\cite{HKBH} and image and signal processing\cite{Comon, Lathauwer, Nagy, Sidiropoulos, Hoge, Rezghi}, etc. Sun et al. \cite{Sun} introduced the notion of the inverse of an even-order tensor under the Einstein product and called it as the Moore-Penrose inverse of tensors. Sun et al. \cite{Sun2} defined the $\{i\}$-inverse and group inverse of tensors based on a general product of tensors, and investigated properties of the generalized inverses of tensors.
Miao et al. \cite{Miao} gave the definition of the generalized tensor function by using the tensor singular value decomposition.  Then, the Cauchy integral formula for tensors were established by taking the advantage of the partial isometry tensors. Moreover, the concept of invariant tensor cones was proposed.
Miao et al. \cite{MQW}  investigated the tensor similar relationship and proposed the T-Jordan canonical
form based on the tensor T-product. Meanwhile, the T-polar, T-LU, T-QR and T-Schur decompositions of tensors were obtained. Besides, the T-group inverse and T-Drazin inverse were studied.
Panigrahy et al. \cite{PBM} studied some
more identities involving the Moore-Penrose inverses of tensors. Also, a few necessary and sufficient conditions of the reverse order
law for the Moore-Penrose inverse of tensors via the Einstein
product were obtained.
Behera et al. \cite{BSMN} researched several generalized inverses of tensors over a commutative ring and a
non-commutative ring. Algorithms for computing the inner inverses, the
Moore-Penrose inverse, and weighted Moore-Penrose inverse of tensors were also proposed. In the final, the application to the image deblurring problem was presented.
Liu et al. \cite{Liu} studied the dual tensor with dual index one based on the T-product. Moreover, the solution of
the dual linear system was presented by taking the advantage of the core inverse of the tensor. The concepts of the dual Moore-Penrose inverse and the group inverse were also established.
Cong et al. \cite{Cong2} established the T-core-EP decomposition of tensors. Moreover, a canonical form and some characterizations of the T-core-EP inverse were given. In the final,  the perturbation bounds for the T-core-EP inverse were studied.
Sahoo et al. \cite{Sahoo} introduced the definitions of the core and the core-EP inverses of the tensors. Some properties, characterizations and representations of the core and the core-EP inverses were given.
Jin et al. \cite{Jin} established the generalized inverse of tensors by using tensor equations. Moreover, the authors investigated the least squares solutions of tensor equations. Behera et al. \cite{Behera} had a further study on the generalized inverses of tensors. Several characterizations of generalized inverses of tensors are provided. Besides, a new method for computing the Moore-Penrose inverse of a tensor was obtained. Ji et al. \cite{JW} extended
the notion of the Drazin inverse of a square matrix to an even-order square tensor. Also, the authors obtained the expression of the Drazin inverse by using the core-nilpotent decomposition. Behera et al. \cite{Behera2} further elaborated the theory of the Drazin inverse and $\wc$-weighted Drazin inverse of tensors. Moreover, different types of methods were built to compute the Drazin inverse of tensors. Ben\'itez et al. \cite{bbj} studied one-sided (b, c)-inverses of arbitrary matrices as well as one-sided inverses along a (not necessarily square) matrix. In addition, the (b, c)-inverse and the inverse along an element were also researched in the context of rectangular matrices. Kolda et al. \cite{KB} provided an overview of higher-order tensor decompositions and their applications. Two particular
tensor decompositions: the CP decomposition and the Tucker decomposition were introduced.

Kernfeld et al. \cite{KKA} defined a new tensor-tensor product---Cosine Transform Product, referred to as C-product for short. And it had been shown that the C-product can be implemented efficiently using DCT. In addition, the authors indicate that one can use C-product to conveniently specify a
discrete image blurring model and the image restoration model.
Xu et al. \cite{XZN} indicated that the
advantages of using DCT are: (a) the complex calculation is not involved in the cosine transform
based singular value decomposition, so the computational costs can be saved; (b) the
intrinsic reflexive boundary condition along the tubes in the third dimension of tensors is employed,
so its performance would be better than that by using the periodic boundary condition in DFT. Moreover, numerical examples showed that the efficiency by using the C-product is two
times faster than that by using the T-product and also the errors of video and multispectral image completion
by using DCT are smaller than those by using DFT. Bentbib et al. \cite{BEJR} explored new applications of the C-product. They proposed new methods for the problem of the third-order tensor
completion in combination with the TV regularization
procedure and tensor robust principal component analysis by using the C-product. Examples are presented to verify the effectiveness of the presented approach. Based on these background, we will study the theory of the generalized inverses of tensors via the C-product in this paper.


This paper is organized as follows. In Section 2, we give the terms and symbols needed to be used in this paper. Then, we introduce the C-product of two tensors and some properties of it. In Section 3, we firstly  define the Moore-Penrose inverse of tensors via the C-product. Then, we provide some decompositions of the tensor, including C-SVD, C-QR decomposition, C-Schur decomposition, C-full rank decomposition, C-QDR decomposition and C-HS decomposition. Furthermore, we use these decompositions to give the expressions for the Moore-Penrose inverse of tensors. In Section 4, we study the Drazin inverse of the tensor under the C-product. This part gives the definition and a few properties for the Drazin inverse of tensors, and provide several expressions for the Drazin inverse of tensors. In Section 5, we define the inverse along a tensor under the C-product. Some expressions of the class of the inverse are obtained. Moreover, an algorithm for computing the inverse along a tensor  is built.
In the last section, we establish an application on higher-order Markov Chains concerning the group inverse of the tensor.
\section{Preliminaries}\label{two}

In this paper, we denote vectors, matrices, three or higher order tensors like $\mathbf{a}, \mathbf{A}, \mathcal{A}$, respectively. Also, $\mathbf{a}_i$, $\mathbf{A}_{ij}$ and $\mathcal{{A}}_{i_1i_2...i_p}$ are the elements of the vector $\mathbf{a}$, matrix $\mathbf{A}$ and tensor $\mathcal{A}$, respectively. The frontal slice of tensor $\mathcal{A}$ is $\mathcal{A}(:, :, i)$. We denote the frontal slice as $\mathcal{A}^{(i)}$
for simplicity. When fixing two indices of the third order tensor, we can get the fiber. The mode-3
fiber is also called tube, denoted as $\mathcal{A}(i, j, :)$. We denote $\overline{\mathbf{a}}$ the tube of the tensor $\mathcal{A}$.  We can vectorize a tube by $\mathbf{a}=\vecc(\overline{\mathbf{a}})$.

\subsection{C-product}

\begin{definition}{\rm\cite{KKA}}
Let $\ac \in \CC^{n_1\times n_2\times n_3}$ and
$\bc \in \CC^{n_2 \times l \times n_3}$. The \textbf{face-wise product} $\mathcal{A}\triangle\mathcal{B}$ is defined as
$$(\mathcal{A}\triangle\mathcal{B})^{(i)}=\mathcal{A}^{(i)}\mathcal{B}^{(i)}.$$
\end{definition}

\begin{definition}{\rm\cite{KKA}}
Let $\ac \in \CC^{n_1\times n_2\times n_3}$. $\mathcal{A}^{(1)}, \mathcal{A}^{(2)},...,\mathcal{A}^{(n_3)}$ are its frontal slices. Then we
 use $\matt(\mathcal{A})$ to denote the block Toeplitz-plus-Hankel matrix
\begin{equation}\label{2}
  \matt(\mathcal{A})=\left[
                       \begin{array}{ccccc}
                         \mathcal{A}^{(1)} & \mathcal{A}^{(2)} & \ldots &\mathcal{A}^{(n_3-1)} &\mathcal{A}^{(n_3)} \\
                         \mathcal{A}^{(2)} & \mathcal{A}^{(1)} & \ldots &\mathcal{A}^{(n_3-2)} &\mathcal{A}^{(n_3-1)} \\
                         \vdots & \vdots &  &\vdots &\vdots \\
                         \mathcal{A}^{(n_3-1)} & \mathcal{A}^{(n_3-2)} & \ldots &\mathcal{A}^{(1)}
                         & \mathcal{A}^{(2)} \\
                         \mathcal{A}^{(n_3)} & \mathcal{A}^{(n_3-1)} & \ldots &\mathcal{A}^{(2)}
                         & \mathcal{A}^{(1)} \\
                       \end{array}
                     \right]+
\left[
                       \begin{array}{ccccc}
                         \mathcal{A}^{(2)} & \mathcal{A}^{(3)} & \ldots & \mathcal{A}^{(n_3)} & O \\
                         \mathcal{A}^{(3)} & \mathcal{A}^{(4)} & \ldots & O & \mathcal{A}^{(n_3)}\\
                         \vdots &\vdots &  & \vdots & \vdots \\
                         \mathcal{A}^{(n_3)} & O &\ldots & \mathcal{A}^{(4)} & \mathcal{A}^{(3)} \\
                         O & \mathcal{A}^{(n_3)} &\ldots  & \mathcal{A}^{(3)} & \mathcal{A}^{(2)} \\
                       \end{array}
                     \right],
\end{equation}
where O is $n_1\times n_2$ zero matrix.
\end{definition}

\begin{definition}{\rm\cite{KKA}}
Let $\ten(\cdot)$ be the inverse operation of the $\matt(\cdot)$, i.e.,
\begin{equation*}
  \ten(\matt(\mathcal{A}))=\mathcal{A}.
\end{equation*}
\end{definition}

\begin{definition}{\rm\cite{KKA}}
Let $\ac \in \CC^{n_1\times n_2\times n_3}$ and $\bc \in \CC^{n_2 \times l \times n_3}$. The
\textbf{cosine transform product}, which is called C-product for short, is defined as
\begin{equation*}
  \ac*_c\bc=\ten(\matt(\mathcal{A})\matt(\mathcal{B})).
\end{equation*}
\end{definition}

Let $\mathbf{\overline{y}}$ be a $1\times1\times n_3$ tensor, then $\matt(\mathbf{\overline{y}})$ is a $1\cdot n_3\times1\cdot n_3$ Toeplitz-plus-Hankel matrix as defined in (\ref{2}), which each blocks are $1\times1$. Let $\mathbf{C}_{n_3}$ denote the $n_3\times n_3$ orthogonal DCT matrix defined in \cite{NCT}, which can be computed in Matlab by using $\mathbf{C}_{n_3}=\dct(\eye(n_3))$. Moreover, one has
\begin{equation*}
  \mathbf{C}_{n_3}\matt(\mathbf{\overline{y}})\mathbf{C}_{n_3}^T=\mathbf{D}=\diag(\mathbf{d}),
\end{equation*}
where $\mathbf{d}=\mathbf{W}^{-1}(\mathbf{C}_{n_3}\matt(\mathbf{\overline{y}})\mathbf{e}_1)$, $\mathbf{W}=\diag(\mathbf{C}_{n_3}(:,1))$, $\mathbf{e}_1=[1,0,...,0]^T$.

Notice that, $\matt(\mathbf{\overline{y}})\mathbf{e}_1=(\mathbf{I}+\mathbf{Z})\vecc(\mathbf{\overline{y}})$, where $\vecc(\mathbf{\overline{y}})$ means the vectorization of $\mathbf{\overline{y}}$, $\mathbf{Z}$ is the $n_3\times n_3$ singular circulant upshift matrix, which can be computed in Matlab by using  $\mathbf{Z}=\diag(\ones(n_3-1,1),1)$. Hence, we have
\begin{equation}\label{l201}
  \mathbf{d}=\mathbf{W}^{-1}\mathbf{C}_{n_3}(\mathbf{I}+\mathbf{Z})\vecc(\mathbf{\overline{y}})
  =\mathbf{M}\vecc(\mathbf{\overline{y}}).
\end{equation}

\begin{definition}\label{d005}{\rm\cite{KKA}}
Let $L : \CC^{1\times 1\times n_3}\rightarrow\CC^{1\times 1\times n_3}$ is an invertible linear transform. Define $$\vecc(L(\overline{\mathbf{y}}))=\mathbf{M}{\mathbf{y}},$$
where $\mathbf{y}=\vecc(\overline{\mathbf{y}})$,  $\mathbf{M}=\mathbf{W}^{-1}\mathbf{C}_{n_3}(\mathbf{I}+\mathbf{Z})$.
\end{definition}

Notice that an $n_1\times n_2\times n_3$ tensor can be seen as an $n_1\times n_2$ matrix whose (i, j)th element $\mathbf{\overline{a}}_{ij}=(\mathcal{A})_{ij}$ are the tube fibers in $\CC^{1\times1\times n_3}$.

\begin{definition}{\rm\cite{KKA}}
Let $\ac \in \CC^{n_1\times n_2\times n_3}$. Then, $L(\mathcal{A})=\widehat{{\mathcal{A}}}\in\CC^{n_1\times n_2\times n_3}$ with tube fibers
\begin{equation*}
  \mathbf{\widehat{{a}}}_{ij}=(\widehat{\mathcal{A}})_{ij}=L(\overline{\mathbf{{a}}}_{ij}), \ i=1,\ldots,n_1, \
  j=1,\ldots,n_2,
\end{equation*}
where $\overline{\mathbf{{a}}}_{ij}$ are the tube fibers of $\mathcal{A}$.
\end{definition}


\begin{definition}{\rm\cite{KB}}
The mode-3 product of a tensor $\ac \in \CC^{n_1\times n_2\times n_3}$ with a matrix $\mathbf{U}\in\CC^{J\times n_3}$ is denoted by $\mathcal{A}\times_3\mathbf{U}$. More precise, we have
\begin{equation*}
  (\mathcal{A}\times_3\mathbf{U})_{i_1i_2j}=
  \sum\limits_{i_3=1}^{n_3}\mathcal{{A}}_{i_1i_2i_3}\mathbf{U}_{ji_3}, \ i_1=1,\ldots,n_1, \ i_2=1,\ldots,n_2, \ j=1,\ldots,J.
\end{equation*}
\end{definition}
Let the frontal slice of $\mathcal{A}\in \CC^{n_1\times n_2\times n_3}$ are
\begin{equation*}
  \mathcal{A}^{(1)}=\left[
                      \begin{array}{cccc}
                        \mathcal{A}_{111} & \mathcal{A}_{121} & \cdots & \mathcal{A}_{1n_21} \\
                        \mathcal{A}_{211} & \mathcal{A}_{221} & \cdots & \mathcal{A}_{2n_21} \\
                        \vdots & \vdots &  & \vdots \\
                        \mathcal{A}_{n_111} & \mathcal{A}_{n_121} & \cdots & \mathcal{A}_{n_1n_21} \\
                      \end{array}
                    \right], \ldots, \mathcal{A}^{(n_3)}=\left[
                      \begin{array}{cccc}
                        \mathcal{A}_{11n_3} & \mathcal{A}_{12n_3} & \cdots & \mathcal{A}_{1n_2n_3} \\
                        \mathcal{A}_{21n_3} & \mathcal{A}_{22n_3} & \cdots & \mathcal{A}_{2n_2n_3} \\
                        \vdots & \vdots &  & \vdots \\
                        \mathcal{A}_{n_11n_3} & \mathcal{A}_{n_12n_3} & \cdots & \mathcal{A}_{n_1n_2n_3} \\
                      \end{array}
                    \right].
\end{equation*}
Then, the mode-3 unfolding of $\mathcal{A}$, denoted $\mathcal{A}_{(3)}$, is
\begin{equation}\label{p1}
  \mathcal{A}_{(3)}=
  \left[
  \begin{array}{ccccccccccccc}
   \mathcal{A}_{111} & \mathcal{A}_{211} & \cdots & \mathcal{A}_{n_111}&
   \mathcal{A}_{121} & \mathcal{A}_{221} & \cdots & \mathcal{A}_{n_121} &\cdots&
   \mathcal{A}_{1n_21}&\mathcal{A}_{2n_21}&\cdots &\mathcal{A}_{n_1n_21}\\
   \mathcal{A}_{112} & \mathcal{A}_{212} & \cdots & \mathcal{A}_{n_112}&
   \mathcal{A}_{122} & \mathcal{A}_{222} & \cdots &\mathcal{A}_{n_122}&\cdots&
   \mathcal{A}_{1n_22}&\mathcal{A}_{2n_22}&\cdots &\mathcal{A}_{n_1n_22}\\
   \vdots & \vdots &  & \vdots &\vdots & \vdots &  & \vdots& &\vdots& \vdots&  &\vdots\\
   \mathcal{A}_{11n_3} & \mathcal{A}_{21n_3} & \cdots & \mathcal{A}_{n_11n_3}&
   \mathcal{A}_{12n_3} & \mathcal{A}_{22n_3} & \cdots & \mathcal{A}_{n_12n_3}&\cdots & \mathcal{A}_{1n_2n_3}&\mathcal{A}_{2n_2n_3}&\cdots & \mathcal{A}_{n_1n_2n_3}\\
                      \end{array}
                    \right].
\end{equation}
Notice that $\mathcal{A}\times_3\mathbf{U}$ can be computed using the following matrix-matrix product. See \cite{KB} for details.
\begin{equation}\label{p2}
  \mathcal{Y}=\mathcal{A}\times_3\mathbf{U}\Leftrightarrow\mathcal{Y}_{(3)}
  =\mathbf{U}\mathcal{A}_{(3)}.
\end{equation}
Observe that
\begin{equation}\label{p3}
  L(\mathcal{A})=\mathcal{A}\times_3\mathbf{M}
\end{equation}
and
\begin{equation}\label{p4}
  L^{-1}(\mathcal{A})=\mathcal{A}\times_3\mathbf{M}^{-1}.
\end{equation}

\begin{lemma}{\rm\cite{KKA}}\label{ll2}
Let $\ac \in \CC^{n_1\times n_2\times n_3}$. Then,
\begin{equation*}
(\mathbf{C}_{n_3}\otimes\mathbf{I}_{n_1})\matt(\mathcal{A})(\mathbf{C}^{-1}_{n_3}\otimes\mathbf{I}_{n_2})  =\left[
   \begin{array}{cccc}
     L(\mathcal{A})^{(1)} &  &  &  \\
      & L(\mathcal{A})^{(2)} &  &  \\
      &  & \ddots &  \\
      &  &  & L(\mathcal{A})^{(n_3)} \\
   \end{array}
 \right],
\end{equation*}
where $\mathbf{C}_{n_3}$ is the $n_3\times n_3$ orthogonal DCT matrix.
\end{lemma}

\begin{lemma}{\rm\cite{KKA}}
Let $\ac \in \CC^{n_1\times n_2\times n_3}$ and
$\bc \in \CC^{n_2 \times l \times n_3}$. Then,
\begin{itemize}
  \item [] $(1)$ $\matt(\mathcal{A}*_c\mathcal{B})=\matt(\mathcal{A})\matt(\mathcal{B})$.
  \item [] $(2)$ $\mathcal{A}*_c\mathcal{B}=L^{-1}(L(\mathcal{A})\triangle L(\mathcal{B}))$.
\end{itemize}
\end{lemma}

The C-product of $\ac \in \CC^{n_1\times n_2\times n_3}$ and
$\bc \in \CC^{n_2 \times l \times n_3}$ can be computed using the following Algorithm borrowed from {\rm\cite{KKA}}.

\begin{algorithm}[H]
\caption{\textsc{Compute the C-product of two tensors }}\label{algo:sylvester9}
\KwIn{$n_1\times n_2\times n_3$ tensor $\mathcal{A}$ and $n_2\times l\times n_3$ tensor $\mathcal{B}$}
\KwOut{$n_1\times l\times n_3$ tensor $\mathcal{C}$}
\begin{enumerate}
\addtolength{\itemsep}{-0.8\parsep minus 0.8\parsep}

\item $\widehat{\mathcal{A}} = L(\mathcal{A})$, $\widehat{\mathcal{B}} = L(\mathcal{B})$

\item

for $i=1,\ldots, n_3$

\quad $\widehat{\mathcal{C}}^{(i)}=\widehat{\mathcal{A}}^{(i)}\widehat{\mathcal{B}}^{(i)}$

end

\item  ${\mathcal{C}} = L^{-1}(\widehat{\mathcal{C}})$
\end{enumerate}
\end{algorithm}

\begin{lemma}{\rm\cite{KKA}}
If $\ac, \bc, \cc$ are order-3 tensors of proper size, then the following statements are true:
\begin{enumerate}[{\rm (1)}]
\item  $\ac*_c(\bc+\cc) = \ac*_c\bc + \ac*_c\cc$;
\item  $(\ac+\bc)*_c\cc = \ac*_c\cc + \bc*_c\cc$;
\item  $(\ac*_c\bc)*_c\cc = \ac*_c(\bc*_c\cc)$.
\end{enumerate}
\end{lemma}

\begin{definition}{\rm\cite{KKA}} 
Let $L(\mathcal{I})=\widehat{\mathcal{I}}\in\CC^{n\times n\times n_3}$ be such that $\widehat{\mathcal{I}}^{(i)}=\mathbf{I}_n$, $i=1,2,...,n_3$. Then $\mathcal{I}= L^{-1}(\widehat{\mathcal{I}})$ is the identity tensor.
\end{definition}

\begin{lemma}{\rm\cite{KKA}}
Let $\ac \in \CC^{n_1\times n_1\times n_3}$ and
$\mathcal{I} \in \CC^{n_1 \times n_1 \times n_3}$ is the identity tensor. Then,
\begin{equation*}
  \mathcal{I}*_c\ac=\ac*_c\mathcal{I}=\mathcal{A}.
\end{equation*}
\end{lemma}
\begin{proof}
It is clear that $$L(\mathcal{I}*_c\ac)=L(\mathcal{I})\triangle L(\ac)=L(\ac)=L(\ac)\triangle L(\mathcal{I})=L(\ac*_c\mathcal{I}).$$ Thus, $\mathcal{I}*_c\ac=\ac*_c\mathcal{I}=\mathcal{A}$.
\end{proof}

\begin{definition}\label{d9}
Let $\ac \in \CC^{n_1\times n_1\times n_3}$ and $\bc \in \CC^{n_1\times n_1\times n_3}$. If
\begin{equation*}\label{0}
  \mathcal{A}*_c\mathcal{B}=\mathcal{I} \ \  \text{and}    \  \ \mathcal{B}*_c\mathcal{A}=\mathcal{I},
\end{equation*}
then $\mathcal{A}$ is said to be {\bf{invertible}} and  $\mathcal{B}$ is the {\bf{inverse}} of $\mathcal{A}$, which is denoted by $\mathcal{A}^{-1}$.
\end{definition}

It is easy to see the inverse of a tensor, if exists, is unique. The conjugate transpose of tensors can be defined as follows.

\begin{definition}{\rm \cite{KKA}} 
If $\ac \in \CC^{n_1 \times n_2 \times n_3}$, then the {\bf{conjugate transpose}} of $\ac$, which is
denoted by $\ac^H$, is such that $$L(\ac^H)^{(i)}=(L(\ac)^{(i)})^H, \ \ i=1,2,...,n_3.$$
\end{definition}

\begin{lemma}{\rm \cite{KKA}} 
Let $\ac \in \CC^{n_1\times n_2\times n_3}$ and $\bc \in \CC^{n_2\times l\times n_3}$. It holds that
\begin{equation*} 
(\ac*_c\bc)^H = \bc^H*_c\ac^H.
\end{equation*}
\end{lemma}

\begin{definition}\label{d3}
Let $\ac \in \CC^{n_1\times n_1\times n_3}$. $\ac$ is said {\bf{symmetric}}
if $\ac^H = \ac$.
\end{definition}

\begin{definition} {\rm \cite{KKA}}
Let $\qc \in \CC^{n_1\times n_1\times n_3}$. $\qc$ is said {\bf{unitary}} if
$\qc^H*_c\qc = \qc*_c\qc^H=\ic$.
\end{definition}



\begin{definition}
Let $\ac \in \CC^{n_1\times n_2\times n_3}$. Then, $\ac$ is called
an {\bf{F-diagonal}}$/${\bf{F-upper}}$/${\bf{F-lower}}  tensor if all frontal slices $\mathcal{A}^{(i)}$, $i=1,2,...,n_3$ of $\mathcal{A}$ are diagonal$/$upper  triangular$/$lower  triangular matrices.
\end{definition}



\begin{lemma}
Let $\ac \in \CC^{n_1\times n_2\times n_3}$. Then, $L(\mathcal{A})$ is an F-diagonal$/$F-upper$/$F-lower tensor if and only if $\ac$ is an F-diagonal$/$F-upper$/$F-lower tensor.
\end{lemma}
\begin{proof}
We only prove the case of the F-lower tensor for the sake of the F-diagonal tensor is one special case of the F-lower tensor and the F-upper tensor can be proved similarly.

Let $\mathcal{B}=L(\mathcal{A})$. Then, by using (\ref{p2}) and (\ref{p4}), one has $\mathcal{A}=L^{-1}(\mathcal{B})=\mathcal{B}\times_3\mathbf{M}^{-1}$ and $\mathcal{A}_{(3)}=\mathbf{M}^{-1}\mathcal{B}_{(3)}$, where $\mathbf{M}$ is defined in (\ref{l201}). Since $\mathcal{B}$ is an F-lower tensor, by (\ref{p1}), one has
\begin{equation*}
  \mathcal{B}_{(3)}=
  \left[
  \begin{array}{cccccccccccccc}
   \mathcal{B}_{111}   & \mathcal{B}_{211}  & \cdots & \mathcal{B}_{n_111}& 0 & \mathcal{B}_{221} & \cdots & \mathcal{B}_{n_121}& 0 & 0 & \mathcal{B}_{331} & \cdots & \mathcal{B}_{n_131}&\cdots \\
   \mathcal{B}_{112}   & \mathcal{B}_{212}  & \cdots & \mathcal{B}_{n_112}& 0 & \mathcal{B}_{222} & \cdots & \mathcal{B}_{n_122}& 0 & 0 & \mathcal{B}_{332} & \cdots & \mathcal{B}_{n_132}&\cdots\\
   \vdots & \vdots &   & \vdots &\vdots & \vdots &  & \vdots& \vdots& \vdots&\vdots& &\vdots\\
   \mathcal{B}_{11n_3} & \mathcal{B}_{21n_3}& \cdots& \mathcal{B}_{n_11n_3}& 0&\mathcal{B}_{22n_3}& \cdots & \mathcal{B}_{n_12n_3}&0 & 0 & \mathcal{B}_{33n_3}&\cdots & \mathcal{B}_{n_13n_3}&\cdots\\
                      \end{array}
                    \right].
\end{equation*}
By using the matrices product, it is easy to see
\begin{equation*}
  \mathcal{A}_{(3)}=
  \left[
  \begin{array}{cccccccccccccc}
   \mathcal{A}_{111}   & \mathcal{A}_{211}  & \cdots & \mathcal{A}_{n_111}& 0 & \mathcal{A}_{221} & \cdots & \mathcal{A}_{n_121}& 0 & 0 & \mathcal{A}_{331} & \cdots & \mathcal{A}_{n_131}&\cdots \\
   \mathcal{A}_{112}   & \mathcal{A}_{212}  & \cdots & \mathcal{A}_{n_112}& 0 & \mathcal{A}_{222} & \cdots & \mathcal{A}_{n_122}& 0 & 0 & \mathcal{A}_{332} & \cdots & \mathcal{A}_{n_132}&\cdots\\
   \vdots & \vdots &   & \vdots &\vdots & \vdots &  & \vdots& \vdots& \vdots&\vdots& &\vdots\\
   \mathcal{A}_{11n_3} & \mathcal{A}_{21n_3}& \cdots& \mathcal{A}_{n_11n_3}& 0&\mathcal{A}_{22n_3}& \cdots & \mathcal{A}_{n_12n_3}&0 & 0 & \mathcal{A}_{33n_3}&\cdots & \mathcal{A}_{n_13n_3}&\cdots\\
                      \end{array}
                    \right].
\end{equation*}
Then, we have all the frontal slices of $\mathcal{A}$ are lower  triangular matrices, which means $\ac$ is an F-lower tensor.

Conversely, if $\ac$ is an F-lower tensor, then $\mathcal{A}_{(3)}$ has the above form. Also, we have $$L(\mathcal{A})=\mathcal{A}\times_3\mathbf{M}\Leftrightarrow L(\mathcal{A})_{(3)}=\mathbf{M}\mathcal{A}_{(3)}$$
 by (\ref{p2}) and (\ref{p3}). Then,
\begin{equation*}
L(\mathcal{A})_{(3)}=\left[
\begin{array}{cccccccccccccc}
 \mathcal{L}_{111}   & \mathcal{L}_{211}  & \cdots & \mathcal{L}_{n_111}& 0 & \mathcal{L}_{221} & \cdots & \mathcal{L}_{n_121}& 0 & 0 & \mathcal{L}_{331} & \cdots & \mathcal{L}_{n_131}&\cdots \\
 \mathcal{L}_{112}   & \mathcal{L}_{212}  & \cdots & \mathcal{L}_{n_112}& 0 & \mathcal{L}_{222} & \cdots & \mathcal{L}_{n_122}& 0 & 0 & \mathcal{L}_{332} & \cdots & \mathcal{L}_{n_132}&\cdots\\
 \vdots & \vdots &   & \vdots &\vdots & \vdots &  & \vdots& \vdots& \vdots&\vdots& &\vdots\\
 \mathcal{L}_{11n_3} & \mathcal{L}_{21n_3}& \cdots& \mathcal{L}_{n_11n_3}& 0&\mathcal{L}_{22n_3}& \cdots & \mathcal{L}_{n_12n_3}&0 & 0 & \mathcal{L}_{33n_3}&\cdots & \mathcal{L}_{n_13n_3}&\cdots\\
\end{array}\right],
\end{equation*}
which implies $L(\mathcal{A})$ is an F-lower tensor.
\end{proof}





\section{The Moore-Penrose inverse of tensors under the C-product}\label{three}

In this part, we will give some expressions of the Moore-Penrose inverse by using the C-SVD, C-QR decomposition, C-Schur decomposition, C-full rank decomposition, C-QDR decomposition and C-HS decomposition. Then, we establish an algorithm to compute the Moore-Penrose inverse based on the C-SVD of a tensor $\mathcal{A}$.

\subsection{The expressions of the Moore-Penrose inverse of tensors}

\begin{definition}
Let $\ac \in \CC^{n_1\times n_2\times n_3}$. If there exists a tensor $\xc \in
\CC^{n_2\times n_1\times n_3}$ such that
\begin{equation}\label{mp}
\ac*_c\xc*_c\ac = \ac, \qquad \xc*_c\ac*_c\xc = \xc, \qquad (\ac*_c\xc)^H = \ac*_c\xc, \qquad (\xc*_c\ac)^H = \xc*_c\ac,
\end{equation}
then $\xc$ is called the \textbf{Moore-Penrose inverse} of the tensor $\ac$ and is denoted by $\ac^\dag$.
\end{definition}

For $\mathcal{A}\in\CC^{n_1\times n_2\times n_3}$,  denote $\mathcal{A}{\{i, j,  \ldots,  k\}}$
the set of all $\mathcal{X}\in\CC^{n_2\times n_1\times n_3}$ which satisfy equations ($i$),  ($j$), $\ldots$ , ($k$) of $(\ref{mp})$.  In this case, $\mathcal{X}$ is a $\{i, j, \ldots, k\}$-inverse.

\begin{theorem}{\rm \cite{KKA}}
Let $\mathcal{A}\in\CC^{n_1\times n_2\times n_3}$. Then there exist
unitary tensors $\mathcal{U}\in\CC^{n_1\times n_1\times n_3}$ and $\mathcal{V}\in\CC^{n_2\times n_2\times n_3}$
such that
\begin{equation*}
  \mathcal{A}=\mathcal{U}*_c\mathcal{S}*_c\mathcal{V}^H,
\end{equation*}
where $\mathcal{S}$ is an $n_1\times n_2\times n_3$ {F-diagonal}  tensor. We call this decomposition the C-SVD of $\mathcal{A}$.
\end{theorem}

\begin{theorem}\label{t2}
The Moore-Penrose inverse of an arbitrary tensor $\mathcal{A}\in\CC^{n_1\times n_2\times n_3}$ exists and is unique.
\end{theorem}
\begin{proof} By Lemma \ref{ll2}, one has
\begin{equation*}
(\mathbf{C}_{n_3}\otimes\mathbf{I}_{n_1})\matt(\mathcal{A})(\mathbf{C}^{-1}_{n_3}\otimes\mathbf{I}_{n_2})  =\left[
   \begin{array}{cccc}
     L(\mathcal{A})^{(1)} &  &  &  \\
      & L(\mathcal{A})^{(2)} &  &  \\
      &  & \ddots &  \\
      &  &  & L(\mathcal{A})^{(n_3)} \\
   \end{array}
 \right].
\end{equation*}
Let $L(\mathcal{A})^{(i)}=\mathbf{U}_i\mathbf{\Sigma}_i\mathbf{V}^H_i$ be the singular value decomposition of $L(\mathcal{A})^{(i)}$, $i=1,...,n_3$. Thus, we have
\begin{eqnarray*}
(\mathbf{C}_{n_3}\otimes\mathbf{I}_{n_1})\matt(\mathcal{A})(\mathbf{C}^{-1}_{n_3}\otimes\mathbf{I}_{n_2})  &=&\left[
   \begin{array}{cccc}
     L(\mathcal{A})^{(1)} &  &  &  \\
      & L(\mathcal{A})^{(2)} &  &  \\
      &  & \ddots &  \\
      &  &  & L(\mathcal{A})^{(n_3)} \\
   \end{array}
 \right]\\
 &=&\left[
   \begin{array}{cccc}
     \mathbf{U}_1\mathbf{\Sigma}_1\mathbf{V}^H_1 &  &  &  \\
      & \mathbf{U}_2\mathbf{\Sigma}_2\mathbf{V}^H_2 &  &  \\
      &  & \ddots &  \\
      &  &  & \mathbf{U}_{n_3}\mathbf{\Sigma}_{n_3}\mathbf{V}^H_{n_3} \\
   \end{array}
 \right].
\end{eqnarray*}
For each $$\mathbf{\Sigma}_i=\left[
                                   \begin{array}{ccccccc}
                                     \sigma^i_1 &  &  & && &\\
                                      &  \ddots&  & & &&\\
                                      &  & & \sigma^i_{r_i} &&&\\
                                      &  &  & &0&&\\
                                      &&&&&\ddots&\\
                                      &&&&&&0\\
                                   \end{array}
                                 \right],$$
$\sigma^i_j, j=1,2,...,{r_i}$, $r_i=rank(L(\mathcal{A})^{(i)})$ are singular values of $L(\mathcal{A})^{(i)}$. We define the matrices $\mathbf{R}_i$, $i=1,...,n_3$, as
$$\mathbf{R}_i=\left[
                                   \begin{array}{ccccccc}
                                     \frac{1}{\sigma^i_1} &  &  & && &\\
                                      &  \ddots&  & & &&\\
                                      &  & & \frac{1}{\sigma^i_{r_i}} &&&\\
                                      &  &  & &0&&\\
                                      &&&&&\ddots&\\
                                      &&&&&&0\\
                                   \end{array}
                                 \right].$$
Observe that $\mathbf{R}_i = \mathbf{\Sigma}_i^\dag$ for $i=1, \ldots, n_3$. Let
$\mathbf{X}_i = \mathbf{V}_i \mathbf{R}_i \mathbf{U}_i^H$ for $i=1, \ldots, n_3$. Now, we have
\begin{equation*}\label{svdd}
  \begin{bmatrix}
                     \mathbf{X}_1 &   &   \\
                       & \ddots   &   \\
                       &   &  \mathbf{X}_{n_3} \\
                   \end{bmatrix}=\begin{bmatrix}
                     \mathbf{V}_1 &   &   \\
                       & \ddots   &   \\
                       &   &  \mathbf{V}_{n_3} \\
                   \end{bmatrix}\begin{bmatrix}
                     \mathbf{R}_1 &   &   \\
                       & \ddots   &   \\
                       &   &  \mathbf{R}_{n_3} \\
                   \end{bmatrix}\begin{bmatrix}
                     \mathbf{U}^H_1 &   &   \\
                       & \ddots   &   \\
                       &   &  \mathbf{U}^H_{n_3} \\
                   \end{bmatrix}.
\end{equation*}
Thus,
\begin{eqnarray*}
 \ten((\mathbf{C}^{-1}_{n_3}\otimes\mathbf{I}_{n_2})\begin{bmatrix}
                     \mathbf{X}_1 &   &   \\
                       & \ddots   &   \\
                       &   &  \mathbf{X}_{n_3} \\
                   \end{bmatrix}(\mathbf{C}_{n_3}\otimes\mathbf{I}_{n_1}))
                   &=&\ten((\mathbf{C}^{-1}_{n_3}\otimes\mathbf{I}_{n_2})\begin{bmatrix}
                     \mathbf{V}_1 &   &   \\
                       & \ddots   &   \\
                       &   &  \mathbf{V}_{n_3} \\
                   \end{bmatrix}(\mathbf{C}_{n_3}\otimes\mathbf{I}_{n_1}))\\
                 & \times& \ten((\mathbf{C}^{-1}_{n_3}\otimes\mathbf{I}_{n_2})\begin{bmatrix}
                     \mathbf{R}_1 &   &   \\
                       & \ddots   &   \\
                       &   &  \mathbf{R}_{n_3} \\
                   \end{bmatrix}(\mathbf{C}_{n_3}\otimes\mathbf{I}_{n_1}))\\
                   & \times& \ten((\mathbf{C}^{-1}_{n_3}\otimes\mathbf{I}_{n_2})\begin{bmatrix}
                     \mathbf{U}^H_1 &   &   \\
                       & \ddots   &   \\
                       &   &  \mathbf{U}^H_{n_3} \\
                   \end{bmatrix}(\mathbf{C}_{n_3}\otimes\mathbf{I}_{n_1})),
\end{eqnarray*}
that is $\mathcal{X}=\mathcal{V}*_c\mathcal{R}*_c\mathcal{U}^H$. It is easy to check that $\mathcal{X}$ satisfies $(\ref{mp})$, which means the Moore-Penrose inverse of a tensor $\mathcal{A}$ exists.

On the other hand, suppose $\mathcal{X}_1$ and $\mathcal{X}_2$ both are the  solutions of $(\ref{mp})$. Then, we have
\begin{eqnarray*}
  \mathcal{X}_1 &=& \mathcal{X}_1*_c\mathcal{A}*_c\mathcal{X}_1=
  \mathcal{X}_1*_c(\mathcal{A}*_c\mathcal{X}_2*_c\mathcal{A})*_c\mathcal{X}_1=\mathcal{X}_1
  *_c(\mathcal{A}*_c\mathcal{X}_2)^H*_c(\mathcal{A}*_c\mathcal{X}_1)^H \\
   &=& \mathcal{X}_1*_c(\mathcal{A}*_c\mathcal{X}_1*_c\mathcal{A}*_c\mathcal{X}_2)^H
   =\mathcal{X}_1*_c(\mathcal{A}*_c\mathcal{X}_2)^H \\
   &=& \mathcal{X}_1*_c\mathcal{A}*_c\mathcal{X}_2\\
   &=&  \mathcal{X}_1*_c(\mathcal{A}*_c\mathcal{X}_2*_c\mathcal{A})*_c\mathcal{X}_2
   =(\mathcal{X}_1*_c\mathcal{A})^H*_c(\mathcal{X}_2*_c\mathcal{A})^H*_c\mathcal{X}_2\\
   &=&  (\mathcal{X}_2*_c\mathcal{A}*_c\mathcal{X}_1*_c\mathcal{A})^H*_c\mathcal{X}_2
   =(\mathcal{X}_2*_c\mathcal{A})^H*_c\mathcal{X}_2\\
   &=& \mathcal{X}_2*_c\mathcal{A}*_c\mathcal{X}_2=\mathcal{X}_2.
\end{eqnarray*}
Therefore, the Moore-Penrose inverse of $\mathcal{A}$ is unique.
\end{proof}

\begin{theorem}
Let $\mathcal{A}\in\CC^{n_1\times n_2\times n_3}$ and $\mathcal{A}=\mathcal{U}*_c\mathcal{S}*_c\mathcal{V}^H$ be the C-SVD of $\mathcal{A}$.
Then,
\begin{equation*}
  \mathcal{A}^\dag=\mathcal{V}*_c\mathcal{S}^\dag*_c\mathcal{U}^H.
\end{equation*}
\end{theorem}
\begin{proof}
It is easy to check that $\mathcal{V}*_c\mathcal{S}^\dag*_c\mathcal{U}^H$ holds for the four equations of $(\ref{mp})$.
\end{proof}

\begin{theorem}\label{t33}
Let $\mathcal{A}\in\CC^{n_1\times n_2\times n_3}$. Then there exist
a unitary tensor $\mathcal{Q}\in\CC^{n_1\times n_1\times n_3}$ and an F-upper  tensor $\mathcal{R}\in\CC^{n_1\times n_2\times n_3}$
such that
\begin{equation*}
  \mathcal{A}=\mathcal{Q}*_c\mathcal{R},
\end{equation*}
which is called the C-QR decomposition of $\mathcal{A}$.
\end{theorem}
\begin{proof}
Let $\widehat{\mathcal{A}}=L(\mathcal{A})$, $\widehat{\mathcal{Q}}=L(\mathcal{Q})$ and $\widehat{\mathcal{R}}=L(\mathcal{R})$. Suppose $\widehat{\mathcal{A}}^{(i)}=\mathbf{Q}_i\mathbf{R}_i=\widehat{\mathcal{Q}}^{(i)}\widehat{\mathcal{R}}^{(i)}$, $i=1,2,...,n_3$, are the QR decomposition of $\widehat{\mathcal{A}}^{(i)}$. Hence, $\mathcal{A}=\mathcal{Q}*_c\mathcal{R}$. Furthermore, one has $L(\mathcal{Q}*_c\mathcal{Q}^H)=L(\mathcal{Q})\triangle L(\mathcal{Q}^H)$. Thus,
\begin{equation*}
L(\mathcal{Q})^{(i)}L(\mathcal{Q}^H)^{(i)}=\widehat{\mathcal{Q}}^{(i)}(\widehat{\mathcal{Q}}^{(i)})^H
  =\mathbf{I}_{n_1}=L(\mathcal{I})^{(i)}, \ i=1,2,...,n_3.
\end{equation*}
This implies $\mathcal{Q}*_c\mathcal{Q}^H=\mathcal{I}$, that is $\mathcal{Q}$ is a unitary tensor. On the other hand, $\mathbf{R}_i$ are upper  triangular matrices and so are $\widehat{\mathcal{R}}^{(i)}$. This implies $\mathcal{R}$ is an F-upper tensor.
\end{proof}

\begin{theorem}
Let $\mathcal{A}\in\CC^{n_1\times n_2\times n_3}$ and $\mathcal{A}=\mathcal{Q}*_c\mathcal{R}$ be the C-QR decomposition of $\mathcal{A}$.
Then,
\begin{equation*}
  \mathcal{A}^\dag=\mathcal{R}^\dag*_c\mathcal{Q}^H.
\end{equation*}
\end{theorem}
\begin{proof}
It is easy to check that $\mathcal{R}^\dag*_c\mathcal{Q}^H$ holds for the four equations of $(\ref{mp})$.
\end{proof}

\begin{theorem}
Let $\mathcal{A}\in\CC^{n\times n\times n_3}$. Then there exist
a unitary tensor $\mathcal{Q}\in\CC^{n\times n\times n_3}$ and an F-upper  tensor $\mathcal{T}\in\CC^{n\times n\times n_3}$
such that
\begin{equation*}
  \mathcal{A}=\mathcal{Q}^H*_c\mathcal{T}*_c\mathcal{Q},
\end{equation*}
which is called the C-Schur decomposition of $\mathcal{A}$.
\end{theorem}
\begin{proof}
Let $\widehat{\mathcal{A}}=L(\mathcal{A})$, $\widehat{\mathcal{Q}}=L(\mathcal{Q})$ and $\widehat{\mathcal{T}}=L(\mathcal{T})$. Suppose $\widehat{\mathcal{A}}^{(i)}=\mathbf{Q}^H_i\mathbf{T}_i\mathbf{Q}_i
=(\widehat{\mathcal{Q}}^{(i)})^H\widehat{\mathcal{T}}^{(i)}\widehat{\mathcal{Q}}^{(i)}$, $i=1,2,...,n_3$, are the Schur decomposition of $\widehat{\mathcal{A}}^{(i)}$. Thus, $\mathcal{A}=\mathcal{Q}^H*_c\mathcal{T}*_c\mathcal{Q}$. By the proof of Theorem \ref{t33}, $\mathcal{Q}$ is a unitary tensor. On the other hand, $\mathbf{T}_i$ are upper  triangular matrices and so are $\widehat{\mathcal{T}}^{(i)}$. This implies $\mathcal{T}$ is an F-upper tensor.
\end{proof}

\begin{theorem}
Let $\mathcal{A}\in\CC^{n\times n\times n_3}$ and $\mathcal{A}=\mathcal{Q}^H*_c\mathcal{T}*_c\mathcal{Q}$ be the C-Schur decomposition of $\mathcal{A}$.
Then,
\begin{equation*}
  \mathcal{A}^\dag=\mathcal{Q}^H*_c\mathcal{T}^\dag*_c\mathcal{Q}.
  \end{equation*}
\end{theorem}
\begin{proof}
Now, we will check that $\mathcal{Q}^H*_c\mathcal{T}^\dag*_c\mathcal{Q}$ holds for the four equations of $(\ref{mp})$. Let $\mathcal{X}=\mathcal{Q}^H*_c\mathcal{T}^\dag*_c\mathcal{Q}$, we will have
\begin{eqnarray*}
  \mathcal{A}*_c\mathcal{X}*_c\mathcal{A}=\mathcal{Q}^H*_c\mathcal{T}*_c\mathcal{Q}*_c
  \mathcal{Q}^H*_c\mathcal{T}^\dag*_c\mathcal{Q}*_c
  \mathcal{Q}^H*_c\mathcal{T}*_c\mathcal{Q}=
  \mathcal{Q}^H*_c\mathcal{T}*_c\mathcal{T}^\dag*_c\mathcal{T}*_c\mathcal{Q}=
  \mathcal{Q}^H*_c\mathcal{T}*_c\mathcal{Q}=\mathcal{A},
\end{eqnarray*}
\begin{eqnarray*}
  \mathcal{X}*_c\mathcal{A}*_c\mathcal{X}=\mathcal{Q}^H*_c\mathcal{T}^\dag*_c\mathcal{Q}*_c
  \mathcal{Q}^H*_c\mathcal{T}*_c\mathcal{Q}*_c
  \mathcal{Q}^H*_c\mathcal{T}^\dag*_c\mathcal{Q}=
  \mathcal{Q}^H*_c\mathcal{T}^\dag*_c\mathcal{T}*_c\mathcal{T}^\dag*_c\mathcal{Q}=
  \mathcal{Q}^H*_c\mathcal{T}^\dag*_c\mathcal{Q}=\mathcal{X},
\end{eqnarray*}
\begin{eqnarray*}
  (\mathcal{A}*_c\mathcal{X})^H&=&(\mathcal{Q}^H*_c\mathcal{T}*_c\mathcal{Q}*_c
  \mathcal{Q}^H*_c\mathcal{T}^\dag*_c\mathcal{Q})^H=
  (\mathcal{Q}^H*_c\mathcal{T}*_c\mathcal{T}^\dag*_c\mathcal{Q})^H\\
  &=&\mathcal{Q}^H*_c\mathcal{T}*_c\mathcal{T}^\dag*_c\mathcal{Q}=
  \mathcal{Q}^H*_c\mathcal{T}*_c\mathcal{Q}*_c
  \mathcal{Q}^H*_c\mathcal{T}^\dag*_c\mathcal{Q}
  =\mathcal{A}*_c\mathcal{X}
\end{eqnarray*}
and
\begin{eqnarray*}
  (\mathcal{X}*_c\mathcal{A})^H&=&(
  \mathcal{Q}^H*_c\mathcal{T}^\dag*_c\mathcal{Q}*_c\mathcal{Q}^H*_c\mathcal{T}*_c\mathcal{Q})^H=
  (\mathcal{Q}^H*_c\mathcal{T}^\dag*_c\mathcal{T}*_c\mathcal{Q})^H\\
  &=&\mathcal{Q}^H*_c\mathcal{T}^\dag*_c\mathcal{T}*_c\mathcal{Q}=
  \mathcal{Q}^H*_c\mathcal{T}^\dag*_c\mathcal{Q}*_c
  \mathcal{Q}^H*_c\mathcal{T}*_c\mathcal{Q}
  =\mathcal{X}*_c\mathcal{A}.
\end{eqnarray*}
\end{proof}

From now on, we denote
\begin{equation*}
\DCT(\matt(\mathcal{A}))=(\mathbf{C}_{n_3}\otimes\mathbf{I}_{n_1})\matt(\mathcal{A})(\mathbf{C}^{-1}_{n_3}\otimes\mathbf{I}_{n_2})  =\left[
   \begin{array}{cccc}
     L(\mathcal{A})^{(1)} &  &  &  \\
      & L(\mathcal{A})^{(2)} &  &  \\
      &  & \ddots &  \\
      &  &  & L(\mathcal{A})^{(n_3)} \\
   \end{array}
 \right],
\end{equation*}
and
\begin{equation*}
\ten(\IDCT(\left[
   \begin{array}{cccc}
     L(\mathcal{A})^{(1)} &  &  &  \\
      & L(\mathcal{A})^{(2)} &  &  \\
      &  & \ddots &  \\
      &  &  & L(\mathcal{A})^{(n_3)} \\
   \end{array}
 \right]))=\mathcal{A}.
\end{equation*}

In the following, we give the full rank decomposition of the tensor. Notice that not all the tensors have the full rank decomposition we defined.
\begin{definition}
Let $\mathcal{A}\in\CC^{n_1\times n_2\times n_3}$. If $\mathcal{A}$ can be decomposed into
\begin{equation*}
  \mathcal{A}=\mathcal{M}*_c\mathcal{N},
\end{equation*}
where  $$\mathcal{M}=\ten(\IDCT(\left[
   \begin{array}{ccc}
     \mathbf{M}_1 &  &    \\
      &   \ddots &  \\
       &  & \mathbf{M}_{n_3} \\
   \end{array}
 \right]))\in\CC^{n_1\times r\times n_3}, \ \mathbf{M}_i\in\CC^{n_1\times r}_r, \ i=1,2,...,n_3$$  and $$\mathcal{N}=\ten(\IDCT(\left[
   \begin{array}{ccc}
     \mathbf{N}_1 &  &    \\
      &   \ddots &  \\
      &    & \mathbf{N}_{n_3} \\
   \end{array}
 \right]))\in\CC^{r\times n_2\times n_3}, \ \mathbf{N}_i\in\CC^{r\times n_2}_r, \ i=1,2,...,n_3,$$  then we call this decomposition the C-full rank decomposition of $\mathcal{A}$.
\end{definition}
{\bf Note:}
Let $\widehat{\mathcal{A}}=L(\mathcal{A})$, $\widehat{\mathcal{M}}=L(\mathcal{M})$ and $\widehat{\mathcal{N}}=L(\mathcal{N})$. Suppose $\widehat{\mathcal{A}}^{(i)}=\mathbf{M}_i\mathbf{N}_i=
\widehat{\mathcal{M}}^{(i)}\widehat{\mathcal{N}}^{(i)}$, $\mathbf{M}_i\in\CC^{n_1\times r}_r$,   $\mathbf{N}_i\in\CC^{r\times n_2}_r$,  $i=1,2,...,n_3$, are the full rank decomposition of $\widehat{\mathcal{A}}^{(i)}$. We deduce when  $rank(\widehat{\mathcal{A}}^{(i)})=r$, $i=1,2,...,n_3$, one has the decomposition of the definition established. $\Box$

\begin{theorem}
Let $\mathcal{A}\in\CC^{n_1\times n_2\times n_3}$. Suppose $\mathcal{A}$ has the C-full rank decomposition $\mathcal{A}=\mathcal{M}*_c\mathcal{N}$.
Then,
\begin{equation*}
  \mathcal{A}^\dag=\mathcal{N}^H*_c(\mathcal{M}^H*_c\mathcal{A}*_c\mathcal{N}^H)^{-1}*_c\mathcal{M}^H.
\end{equation*}
\end{theorem}
\begin{proof}
We will check that $\mathcal{N}^H*_c(\mathcal{M}^H*_c\mathcal{A}*_c\mathcal{N}^H)^{-1}*_c\mathcal{M}^H$ holds for the four equations of $(\ref{mp})$. Let $\mathcal{X}=\mathcal{N}^H*_c(\mathcal{M}^H*_c\mathcal{A}*_c\mathcal{N}^H)^{-1}*_c\mathcal{M}^H$. Then, we have
\begin{eqnarray*}
  \mathcal{A}*_c\mathcal{X}*_c\mathcal{A}&=&\mathcal{M}*_c\mathcal{N}*_c
  \mathcal{N}^H*_c(\mathcal{M}^H*_c\mathcal{A}*_c\mathcal{N}^H)^{-1}*_c\mathcal{M}^H*_c\mathcal{M}*_c\mathcal{N}\\
  &=&\mathcal{M}*_c\mathcal{N}*_c
\mathcal{N}^H*_c(\mathcal{M}^H*_c\mathcal{M}*_c\mathcal{N}*_c\mathcal{N}^H)^{-1}*_c\mathcal{M}^H*_c\mathcal{M}*_c\mathcal{N}\\
&=&\mathcal{M}*_c\mathcal{N}*_c
\mathcal{N}^H*_c(\mathcal{N}*_c\mathcal{N}^H)^{-1}*_c(\mathcal{M}^H*_c\mathcal{M})^{-1}*_c\mathcal{M}^H*_c\mathcal{M}*_c\mathcal{N}\\
&=&\mathcal{M}*_c\mathcal{N}=\mathcal{A},
  \end{eqnarray*}
\begin{eqnarray*}
  \mathcal{X}*_c\mathcal{A}*_c\mathcal{X}&=&
\mathcal{N}^H*_c(\mathcal{M}^H*_c\mathcal{A}*_c\mathcal{N}^H)^{-1}*_c\mathcal{M}^H*_c
\mathcal{A}*_c
\mathcal{N}^H*_c(\mathcal{M}^H*_c\mathcal{A}*_c\mathcal{N}^H)^{-1}*_c\mathcal{M}^H\\
&=& \mathcal{N}^H*_c(\mathcal{M}^H*_c\mathcal{A}*_c\mathcal{N}^H)^{-1}\mathcal{M}^H
 =\mathcal{X},
\end{eqnarray*}
\begin{eqnarray*}
  (\mathcal{A}*_c\mathcal{X})^H&=&[\mathcal{M}*_c\mathcal{N}*_c
  \mathcal{N}^H*_c(\mathcal{M}^H*_c\mathcal{A}*_c\mathcal{N}^H)^{-1}*_c\mathcal{M}^H]^H\\
  &=&[\mathcal{M}*_c\mathcal{N}*_c
\mathcal{N}^H*_c(\mathcal{N}*_c\mathcal{N}^H)^{-1}*_c(\mathcal{M}^H*_c\mathcal{M})^{-1}*_c\mathcal{M}^H]^H\\
  &=&\mathcal{M}*_c(\mathcal{M}^H*_c\mathcal{M})^{-1}*_c\mathcal{M}^H\\
 &=&\mathcal{M}*_c\mathcal{N}*_c
\mathcal{N}^H*_c(\mathcal{N}*_c\mathcal{N}^H)^{-1}*_c(\mathcal{M}^H*_c\mathcal{M})^{-1}*_c\mathcal{M}^H\\
&=&\mathcal{M}*_c\mathcal{N}*_c
\mathcal{N}^H*_c(\mathcal{M}^H*_c\mathcal{A}*_c\mathcal{N}^H)^{-1}*_c\mathcal{M}^H
  =\mathcal{A}*_c\mathcal{X}
\end{eqnarray*}
and
\begin{eqnarray*}
  (\mathcal{X}*_c\mathcal{A})^H&=&[
  \mathcal{N}^H*_c(\mathcal{M}^H*_c\mathcal{A}*_c\mathcal{N}^H)^{-1}*_c\mathcal{M}^H*_c\mathcal{M}*_c\mathcal{N}]^H\\
  &=&[
\mathcal{N}^H*_c(\mathcal{N}*_c\mathcal{N}^H)^{-1}*_c(\mathcal{M}^H*_c\mathcal{M})^{-1}*_c\mathcal{M}^H
*_c\mathcal{M}*_c\mathcal{N}]^H\\
  &=&\mathcal{N}^H*_c(\mathcal{N}*_c\mathcal{N}^H)^{-1}*_c\mathcal{N}\\
 &=&\mathcal{N}^H*_c(\mathcal{N}*_c\mathcal{N}^H)^{-1}*_c(\mathcal{M}^H*_c\mathcal{M})^{-1}*_c
 \mathcal{M}^H*_c\mathcal{M}*_c\mathcal{N}\\
&=&\mathcal{N}^H*_c(\mathcal{M}^H*_c\mathcal{A}*_c\mathcal{N}^H)^{-1}*_c\mathcal{M}^H*_c\mathcal{M}*_c\mathcal{N}
 =\mathcal{X}*_c\mathcal{A}.
\end{eqnarray*}
\end{proof}




\begin{definition}
Let $\mathcal{A}\in\CC^{n_1\times n_2\times n_3}$. If $\mathcal{A}$ can be decomposed as
\begin{equation*}
  \mathcal{A}=\mathcal{Q}*_c\mathcal{D}*_c\mathcal{R},
\end{equation*}
where $$\mathcal{Q}=\ten(\IDCT(\left[
   \begin{array}{ccc}
     \mathbf{Q}_1 &  &    \\
      &   \ddots &  \\
       &  & \mathbf{Q}_{n_3} \\
   \end{array}
 \right]))\in\CC^{n_1\times r\times n_3}, \ \mathbf{Q}_i\in\CC^{n_1\times r}_r, \ i=1,2,...,n_3,$$
$\mathcal{D}\in\CC^{r\times r\times n_3}$ is an invertible F-diagonal tensor and $$\mathcal{R}=\ten(\IDCT(\left[
   \begin{array}{ccc}
     \mathbf{R}_1 &  &    \\
      &   \ddots &  \\
       &  & \mathbf{R}_{n_3} \\
   \end{array}
 \right]))\in\CC^{r\times n_2\times n_3}, \ \mathbf{R}_i\in\CC^{r\times n_2}_r, \ i=1,2,...,n_3$$ is an F-upper tensor,
then we call this decomposition the C-QDR decomposition of $\mathcal{A}$.
\end{definition}
{\bf Note:}
Let $\widehat{\mathcal{A}}=L(\mathcal{A})$, $\widehat{\mathcal{Q}}=L(\mathcal{Q})$, $\widehat{\mathcal{D}}=L(\mathcal{D})$ and $\widehat{\mathcal{R}}=L(\mathcal{R})$. Suppose $$\widehat{\mathcal{A}}^{(i)}=\mathbf{Q}_i\mathbf{D}_i\mathbf{R}_i
=\widehat{\mathcal{Q}}^{(i)}\widehat{\mathcal{D}}^{(i)}\widehat{\mathcal{R}}^{(i)}, \ \mathbf{Q}_i\in\CC^{n_1\times r}_r, \ \mathbf{D}_i\in\CC^{r\times r}_r,  \ \mathbf{R}_i\in\CC^{r\times n_2}_r, \ i=1,2,...,n_3,$$ are the QDR decomposition of $\widehat{\mathcal{A}}^{(i)}$ \cite{SPKS}. We deduce when  $rank(\widehat{\mathcal{A}}^{(i)})=r$, $i=1,2,...,n_3$, one has the decomposition the definition  established. Since $\mathbf{D}_i$ are invertible diagonal matrices, we have $\widehat{\mathcal{D}}^{(i)}$ also are invertible diagonal matrices. Meanwhile, $\mathbf{R}_i$ are upper triangular matrices and so are $\widehat{\mathcal{R}}^{(i)}$. This implies $\mathcal{D}$ is an invertible F-diagonal tensor and $\mathcal{R}$ is an F-upper tensor. $\Box$

\begin{theorem}
Let $\mathcal{A}\in\CC^{n_1\times n_2\times n_3}$. Suppose $\mathcal{A}^H$ has the C-QDR decomposition $\mathcal{A}^H=\mathcal{Q}*_c\mathcal{D}*_c\mathcal{R}$.
Then,
\begin{equation*}
  \mathcal{A}^\dag=\mathcal{Q}*_c(\mathcal{R}*_c\mathcal{A}*_c\mathcal{Q})^{-1}*_c\mathcal{R}.
\end{equation*}
\end{theorem}
\begin{proof}
Let $\mathcal{X}=\mathcal{Q}*_c(\mathcal{R}*_c\mathcal{A}*_c\mathcal{Q})^{-1}*_c\mathcal{R}$. Thus, one has
\begin{eqnarray*}
  \mathcal{A}*_c\mathcal{X}*_c\mathcal{A}&=&\mathcal{R}^H*_c\mathcal{D}^H*_c\mathcal{Q}^H*_c
  \mathcal{Q}*_c(\mathcal{R}*_c\mathcal{A}*_c\mathcal{Q})^{-1}*_c\mathcal{R}*_c
 \mathcal{R}^H*_c\mathcal{D}^H*_c\mathcal{Q}^H\\
&=&\mathcal{R}^H*_c\mathcal{D}^H*_c\mathcal{Q}^H*_c
  \mathcal{Q}*_c(\mathcal{R}*_c\mathcal{R}^H*_c\mathcal{D}^H*_c\mathcal{Q}^H*_c\mathcal{Q})^{-1}*_c\mathcal{R}*_c
 \mathcal{R}^H*_c\mathcal{D}^H*_c\mathcal{Q}^H \\
&=&\mathcal{R}^H*_c\mathcal{D}^H*_c\mathcal{Q}^H=\mathcal{A},
  \end{eqnarray*}
\begin{eqnarray*}
  \mathcal{X}*_c\mathcal{A}*_c\mathcal{X}&=&
  \mathcal{Q}*_c(\mathcal{R}*_c\mathcal{A}*_c\mathcal{Q})^{-1}*_c\mathcal{R}*_c
  \mathcal{A}*_c
  \mathcal{Q}*_c(\mathcal{R}*_c\mathcal{A}*_c\mathcal{Q})^{-1}*_c\mathcal{R}
 \\
&=&\mathcal{Q}*_c(\mathcal{R}*_c\mathcal{A}*_c\mathcal{Q})^{-1}*_c\mathcal{R}=\mathcal{X},
  \end{eqnarray*}
\begin{eqnarray*}
  (\mathcal{A}*_c\mathcal{X})^H&=&[\mathcal{R}^H*_c\mathcal{D}^H*_c\mathcal{Q}^H*_c
  \mathcal{Q}*_c(\mathcal{R}*_c\mathcal{A}*_c\mathcal{Q})^{-1}*_c\mathcal{R}]^H\\
  &=&[\mathcal{R}^H*_c\mathcal{D}^H*_c\mathcal{Q}^H*_c
  \mathcal{Q}*_c(\mathcal{Q}^H*_c
  \mathcal{Q})^{-1}*_c(\mathcal{D}^H)^{-1}*_c(\mathcal{R}*_c
  \mathcal{R}^H)^{-1}*_c\mathcal{R}]^H\\
  &=&\mathcal{R}^H*_c(\mathcal{R}*_c\mathcal{R}^H)^{-1}*_c\mathcal{R}\\
 &=&\mathcal{R}^H*_c\mathcal{D}^H*_c\mathcal{Q}^H*_c\mathcal{Q}*_c(\mathcal{Q}^H*_c
  \mathcal{Q})^{-1}*_c(\mathcal{D}^H)^{-1}*_c(\mathcal{R}*_c
  \mathcal{R}^H)^{-1}*_c\mathcal{R}\\
&=&\mathcal{R}^H*_c\mathcal{D}^H*_c\mathcal{Q}^H*_c
  \mathcal{Q}*_c(\mathcal{R}*_c\mathcal{A}*_c\mathcal{Q})^{-1}*_c\mathcal{R}
  =\mathcal{A}*_c\mathcal{X}
\end{eqnarray*}
and
\begin{eqnarray*}
  (\mathcal{X}*_c\mathcal{A})^H&=&[\mathcal{Q}*_c(\mathcal{R}*_c\mathcal{A}*_c\mathcal{Q})^{-1}*_c\mathcal{R}*_c
\mathcal{R}^H*_c\mathcal{D}^H*_c\mathcal{Q}^H]^H\\
  &=&[\mathcal{Q}*_c(\mathcal{Q}^H*_c
  \mathcal{Q})^{-1}*_c(\mathcal{D}^H)^{-1}*_c(\mathcal{R}*_c
  \mathcal{R}^H)^{-1}*_c\mathcal{R}*_c\mathcal{R}^H*_c\mathcal{D}^H*_c\mathcal{Q}^H
]^H\\
  &=&\mathcal{Q}*_c(\mathcal{Q}^H*_c
  \mathcal{Q})^{-1}*_c\mathcal{Q}^H\\
 &=&\mathcal{Q}*_c(\mathcal{Q}^H*_c
  \mathcal{Q})^{-1}*_c(\mathcal{D}^H)^{-1}*_c(\mathcal{R}*_c
  \mathcal{R}^H)^{-1}*_c\mathcal{R}*_c\mathcal{R}^H*_c\mathcal{D}^H*_c\mathcal{Q}^H\\
&=&\mathcal{Q}*_c(\mathcal{R}*_c\mathcal{A}*_c\mathcal{Q})^{-1}*_c\mathcal{R}*_c
\mathcal{R}^H*_c\mathcal{D}^H*_c\mathcal{Q}^H=\mathcal{X}*_c\mathcal{A}.
\end{eqnarray*}
Therefore, $\mathcal{X}=\mathcal{A}^\dag$.
\end{proof}

For a tensor $\mathcal{A}\in\CC^{n_1\times n_2\times n_3}$, which the block form is
$$\mathcal{A}=\left[
\begin{array}{cc}
\mathcal{A}_1 & \mathcal{A}_2 \\
\mathcal{A}_3 & \mathcal{A}_4 \\
\end{array}
\right],$$
\text{where}  $\mathcal{A}_1\in\CC^{s\times t\times n_3}, \mathcal{A}_2\in\CC^{s\times (n_2-t)\times n_3}, \mathcal{A}_3\in\CC^{(n_1-s)\times t\times n_3},
\mathcal{A}_4\in\CC^{(n_1-s)\times(n_2-t)\times n_3}$. Let
$$\mathcal{B}=\left[
\begin{array}{cc}
\mathcal{B}_1 & \mathcal{B}_2 \\
\mathcal{B}_3 & \mathcal{B}_4 \\
\end{array}
\right]\in\CC^{n_2\times n_4\times n_3},$$
\text{where}  $\mathcal{B}_1\in\CC^{t\times k\times n_3}, \mathcal{B}_2\in\CC^{t\times (n_4-k)\times n_3}, \mathcal{B}_3\in\CC^{(n_2-t)\times k\times n_3},
\mathcal{B}_4\in\CC^{(n_2-t)\times(n_4-k)\times n_3}$. It is easy to check that
$$\mathcal{A}*_c\mathcal{B}=\left[
\begin{array}{cc}
\mathcal{A}_1 & \mathcal{A}_2 \\
\mathcal{A}_3 & \mathcal{A}_4 \\
\end{array}
\right]*_c\left[
\begin{array}{cc}
\mathcal{B}_1 & \mathcal{B}_2 \\
\mathcal{B}_3 & \mathcal{B}_4 \\
\end{array}
\right]=\left[
\begin{array}{cc}
\mathcal{A}_1*_c\mathcal{B}_1+\mathcal{A}_2*_c\mathcal{B}_3
& \mathcal{A}_1*_c\mathcal{B}_2+\mathcal{A}_2*_c\mathcal{B}_4 \\
\mathcal{A}_3*_c\mathcal{B}_1+\mathcal{A}_4*_c\mathcal{B}_3
& \mathcal{A}_3*_c\mathcal{B}_2+\mathcal{A}_4*_c\mathcal{B}_4 \\
\end{array}
\right].$$

Suppose $\mathcal{A}\in\CC^{n_1\times n_1\times n_3}$ and
$\mathcal{A}=\mathcal{U}*_c\mathcal{S}*_c\mathcal{V}^H$ is the C-SVD of $\mathcal{A}$. When  $$rank\left(L(\mathcal{S})^{(1)}\right)=rank\left(L(\mathcal{S})^{(2)}\right)=\cdots=rank\left(L(\mathcal{S})^{(n_3)}\right)=r,$$ the decomposition of $\mathcal{A}$ can be written as
$$\mathcal{A}=\mathcal{U}*_c\left[
\begin{array}{cc}
\mathcal{S}_r & \mathcal{O} \\
\mathcal{O} & \mathcal{O} \\
\end{array}
\right]
*_c\mathcal{V}^H,$$
where $\mathcal{S}_r\in\CC^{r\times r\times n_3}$, $\mathcal{U}\in\CC^{n_1\times n_1\times n_3}$, $\mathcal{V}\in\CC^{n_1\times n_1\times n_3}$. Let
$$\mathcal{V}^H*_c\mathcal{U}=\left[
\begin{array}{cc}
\mathcal{K} & \mathcal{L} \\
\mathcal{M} & \mathcal{N} \\
\end{array}
\right], \ \text{where} \ \mathcal{K}\in\CC^{r\times r\times n_3}.$$
Thus, we have
$$\mathcal{A}=\mathcal{U}*_c\left[
\begin{array}{cc}
\mathcal{S}_r & \mathcal{O} \\
\mathcal{O} & \mathcal{O} \\
\end{array}
\right]
*_c\mathcal{V}^H=\mathcal{U}*_c\left[
\begin{array}{cc}
\mathcal{S}_r & \mathcal{O} \\
\mathcal{O} & \mathcal{O} \\
\end{array}
\right]
*_c\left[
\begin{array}{cc}
\mathcal{K} & \mathcal{L} \\
\mathcal{M} & \mathcal{N} \\
\end{array}
\right]*_c\mathcal{U}^H=\mathcal{U}*_c\left[
\begin{array}{cc}
\mathcal{S}_r*_c\mathcal{K} & \mathcal{S}_r*_c\mathcal{L} \\
\mathcal{O} & \mathcal{O} \\
\end{array}
\right]*_c\mathcal{U}^H.$$
Since $\mathcal{V}^H*_c\mathcal{U}$ is unitary, one can arrive $\mathcal{K}*_c\mathcal{K}^H+\mathcal{L}*_c\mathcal{L}^H=\mathcal{I}_r$, where $\mathcal{I}_r\in\CC^{r\times r\times n_3}$. We call this decomposition the C-HS decomposition of $\mathcal{A}$.


\begin{theorem}
Let $\mathcal{A}\in\CC^{n_1\times n_1\times n_3}$. Suppose $\mathcal{A}$ has the C-HS decomposition.
Then,
\begin{equation} \label{a00}
\ac^\dag = \uc *_c
\mat{\mathcal{K}^H*_c\mathcal{S}_r^{-1}}{\mathcal{O}}{\mathcal{L}^H*_c\mathcal{S}_r^{-1}}{\mathcal{O}}*_c\uc^H.
\end{equation}
\end{theorem}

\begin{proof}
Let $\mathcal{A}=\mathcal{U}*_c\left[
\begin{array}{cc}
\mathcal{S}_r*_c\mathcal{K} & \mathcal{S}_r*_c\mathcal{L} \\
\mathcal{O} & \mathcal{O} \\
\end{array}
\right]*_c\mathcal{U}^H$ and $\mathcal{X}=\uc *_c
\mat{\mathcal{K}^H*_c\mathcal{S}_r^{-1}}{\mathcal{O}}{\mathcal{L}^H*_c\mathcal{S}_r^{-1}}{\mathcal{O}}*_c\uc^H$. Then,
\begin{eqnarray*}
  \mathcal{A}*_c\mathcal{X}*_c\mathcal{A}&=&\mathcal{U}*_c\left[
\begin{array}{cc}
\mathcal{S}_r*_c\mathcal{K} & \mathcal{S}_r*_c\mathcal{L} \\
\mathcal{O} & \mathcal{O} \\
\end{array}
\right]*_c\mathcal{U}^H*_c\uc *_c
\mat{\mathcal{K}^H*_c\mathcal{S}_r^{-1}}{\mathcal{O}}{\mathcal{L}^H*_c\mathcal{S}_r^{-1}}{\mathcal{O}}*_c\uc^H*_c\mathcal{U}*_c\left[
\begin{array}{cc}
\mathcal{S}_r*_c\mathcal{K} & \mathcal{S}_r*_c\mathcal{L} \\
\mathcal{O} & \mathcal{O} \\
\end{array}
\right]*_c\mathcal{U}^H \\
   &=&  \mathcal{U}*_c\left[
\begin{array}{cc}
\mathcal{S}_r*_c\mathcal{K} & \mathcal{S}_r*_c\mathcal{L} \\
\mathcal{O} & \mathcal{O} \\
\end{array}
\right]*_c\mathcal{U}^H=\mathcal{A}.
\end{eqnarray*}
\begin{eqnarray*}
  \mathcal{X}*_c\mathcal{A}*_c\mathcal{X} &=& \uc *_c
\mat{\mathcal{K}^H*_c\mathcal{S}_r^{-1}}{\mathcal{O}}{\mathcal{L}^H*_c\mathcal{S}_r^{-1}}{\mathcal{O}}*_c\uc^H
*_c\mathcal{U}*_c\left[
\begin{array}{cc}
\mathcal{S}_r*_c\mathcal{K} & \mathcal{S}_r*_c\mathcal{L} \\
\mathcal{O} & \mathcal{O} \\
\end{array}
\right]*_c\mathcal{U}^H*_c
\uc *_c
\mat{\mathcal{K}^H*_c\mathcal{S}_r^{-1}}{\mathcal{O}}{\mathcal{L}^H*_c\mathcal{S}_r^{-1}}{\mathcal{O}}*_c\uc^H \\
   &=& \uc *_c
\mat{\mathcal{K}^H*_c\mathcal{S}_r^{-1}}{\mathcal{O}}{\mathcal{L}^H*_c\mathcal{S}_r^{-1}}{\mathcal{O}}*_c\uc^H=\mathcal{X}.
\end{eqnarray*}
\begin{eqnarray*}
  \mathcal{A}*_c\mathcal{X}&=&\mathcal{U}*_c\left[
\begin{array}{cc}
\mathcal{S}_r*_c\mathcal{K} & \mathcal{S}_r*_c\mathcal{L} \\
\mathcal{O} & \mathcal{O} \\
\end{array}
\right]*_c\mathcal{U}^H*_c\uc *_c
\mat{\mathcal{K}^H*_c\mathcal{S}_r^{-1}}{\mathcal{O}}{\mathcal{L}^H*_c\mathcal{S}_r^{-1}}{\mathcal{O}}*_c\uc^H\\
   &=&  \mathcal{U}*_c\left[
\begin{array}{cc}
\mathcal{I}_r & \mathcal{O} \\
\mathcal{O} & \mathcal{O} \\
\end{array}
\right]*_c\mathcal{U}^H=(\mathcal{A}*_c\mathcal{X})^H.
\end{eqnarray*}
\begin{eqnarray*}
  \mathcal{X}*_c\mathcal{A}&=&\uc *_c
\mat{\mathcal{K}^H*_c\mathcal{S}_r^{-1}}{\mathcal{O}}{\mathcal{L}^H*_c\mathcal{S}_r^{-1}}{\mathcal{O}}*_c\uc^H*_c\mathcal{U}*_c\left[
\begin{array}{cc}
\mathcal{S}_r*_c\mathcal{K} & \mathcal{S}_r*_c\mathcal{L} \\
\mathcal{O} & \mathcal{O} \\
\end{array}
\right]*_c\mathcal{U}^H\\
   &=&  \mathcal{U}*_c\left[
\begin{array}{cc}
\mathcal{K}^H*_c\mathcal{K} & \mathcal{K}^H*_c\mathcal{L} \\
\mathcal{L}^H*_c\mathcal{K} & \mathcal{L}^H*_c\mathcal{L} \\
\end{array}
\right]*_c\mathcal{U}^H=(\mathcal{X}*_c\mathcal{A})^H.
\end{eqnarray*}
Therefore, $\ac^\dag = \uc *_c
\mat{\mathcal{K}^H*_c\mathcal{S}_r^{-1}}{\mathcal{O}}
{\mathcal{L}^H*_c\mathcal{S}_r^{-1}}{\mathcal{O}}*_c\uc^H$.
\end{proof}

\subsection{The algorithm for computing the  Moore-Penrose inverse of a tensor}

In the following, we have Algorithm \ref{algo11} provided the procedure for the Moore-Penrose inverse operation.

\begin{algorithm}[H]
\caption{\textsc{Compute the Moore-Penrose inverse of a tensor $\mathcal{A}$}}\label{algo11}
\KwIn{$n_1\times n_2\times n_3$ tensor $\mathcal{A}$}
\KwOut{$n_2\times n_1\times n_3$ tensor $\mathcal{X}$}
\begin{enumerate}
\addtolength{\itemsep}{-0.8\parsep minus 0.8\parsep}

\item $\widehat{\mathcal{A}} = L(\mathcal{A})=\mathcal{A}\times_3\mathbf{M}$, where $\mathbf{M}$ is defined in (\ref{l201})

\item  for $i=1,\ldots, n_3$

\qquad  $\widehat{\mathcal{X}}^{(i)}=\pinv(\widehat{\mathcal{A}}^{(i)})$; where
$\pinv(\widehat{\mathcal{A}}^{(i)})$ is the Moore-Penrose inverse of $\widehat{\mathcal{A}}^{(i)}$

end

\item  ${\mathcal{X}} = L^{-1}(\widehat{\mathcal{X}})=\widehat{\mathcal{X}}\times_3\mathbf{M}^{-1}$
\end{enumerate}
\end{algorithm}

\begin{example}
Let $\ac \in \CC^{3\times 3\times 4}$ with frontal slices
$$\ac^{(1)}=\begin{bmatrix}
    1  & 0 & 0\\
    0  & 1 & 0\\
    0  & 0 & 3\\
  \end{bmatrix},\quad
\ac^{(2)}=\begin{bmatrix}
    2  & 3 & 0\\
    2  & 0 & 0\\
    1  & 0 & 5\\
  \end{bmatrix},\quad
\ac^{(3)}=\begin{bmatrix}
    3  & 1 & 0\\
    0  & 2 & 3\\
    4  & 0 & 0\\
  \end{bmatrix},
  \quad
\ac^{(4)}=\begin{bmatrix}
    3  & 1 & 4\\
    0  & 2 & 2\\
    1  & 0 & 2\\
  \end{bmatrix}.$$
Then, by using Algorithm \ref{algo11}, we have
\begin{eqnarray*}
(\ac^\dag)^{(1)}=\begin{bmatrix}
    1.6666  & 1.3333  & 9.7778\\
    1.3333  & 1  & 7.5556\\
    0     & 0     & -0.3333\\
  \end{bmatrix},\quad
(\ac^\dag)^{(2)}=\begin{bmatrix}
    -1.2722  & -1.0482 & -8.2780\\
    -1.2295  & -0.7384 & -6.2015\\
    0.1057  & -0.0651  & 0.2724\\
  \end{bmatrix},
  \end{eqnarray*}
\begin{eqnarray*}
(\ac^\dag)^{(3)}=\begin{bmatrix}
    0.7451  & 0.7255 & 5.0065\\
    1.1372  & 0.3529 & 3.4837\\
    -0.2353 & 0.1568 & -0.0196\\
  \end{bmatrix},\quad
(\ac^\dag)^{(4)}=\begin{bmatrix}
    -0.2723  & -0.3815 & -1.6113\\
    -0.5629  & -0.0718 & -1.0905\\
    0.1057 & -0.0651 & -0.0610\\
  \end{bmatrix}.
\end{eqnarray*}
\end{example}

\section{The Drazin inverse of tensors under the C-product}\label{three1}
In this section, we will give some expressions of the Drazin inverse of tensors. Then, an algorithm is established for the Drazin inverse of a tensor.

\subsection{The expressions of the Drazin inverse of tensors}

Recall that the index of a matrix $A$ is defined as the smallest nonnegative integer $k$ such that
$\rk(A^k) = \rk(A^{k+1})$, which is denoted by $\ind(A)$. Now, let us define the index of a tensor $\mathcal{A}$.

\begin{definition} 
Let $\ac \in \CC^{n_1\times n_1\times n_3}$. The index of the tensor $\mathcal{A}$ is defined as
\begin{eqnarray*} 
\ind(\mathcal{A})=\ind(\matt(\mathcal{A})).
\end{eqnarray*}
\end{definition}

\begin{lemma} \label{d52}
Let $\ac\in \CC^{n_1\times n_1\times n_3}$. Suppose
that $\ac$ can be expressed as
\begin{eqnarray*} 
\DCT(\matt(\mathcal{A}))=\left[
   \begin{array}{cccc}
     L(\mathcal{A})^{(1)} &  &  &  \\
      & L(\mathcal{A})^{(2)} &  &  \\
      &  & \ddots &  \\
      &  &  & L(\mathcal{A})^{(n_3)} \\
   \end{array}
 \right].
\end{eqnarray*}
Then, $\ind(\mathcal{A})=\max\limits_{1\leq i\leq n_3}\{\ind(L(\mathcal{A})^{(i)})\}$.
\end{lemma}

\begin{proof}
Notice that
\begin{equation*}
\matt(\mathcal{A})=(\mathbf{C}^{-1}_{n_3}\otimes\mathbf{I}_{n_1})\left[
   \begin{array}{cccc}
     L(\mathcal{A})^{(1)} &  &  &  \\
      & L(\mathcal{A})^{(2)} &  &  \\
      &  & \ddots &  \\
      &  &  & L(\mathcal{A})^{(n_3)} \\
   \end{array}
 \right](\mathbf{C}_{n_3}\otimes\mathbf{I}_{n_1}).
\end{equation*}
Thus,
\begin{eqnarray*}
(\matt(\mathcal{A}))^k&=&(\mathbf{C}^{-1}_{n_3}\otimes\mathbf{I}_{n_1})\left[
   \begin{array}{cccc}
     L(\mathcal{A})^{(1)} &  &  &  \\
      & L(\mathcal{A})^{(2)} &  &  \\
      &  & \ddots &  \\
      &  &  & L(\mathcal{A})^{(n_3)} \\
   \end{array}
 \right]^k(\mathbf{C}_{n_3}\otimes\mathbf{I}_{n_1})\\
&=&(\mathbf{C}^{-1}_{n_3}\otimes\mathbf{I}_{n_1})\left[
   \begin{array}{cccc}
     (L(\mathcal{A})^{(1)})^k &  &  &  \\
      & (L(\mathcal{A})^{(2)})^k &  &  \\
      &  & \ddots &  \\
      &  &  & (L(\mathcal{A})^{(n_3)})^k \\
   \end{array}
 \right](\mathbf{C}_{n_3}\otimes\mathbf{I}_{n_1}) ,
\end{eqnarray*}
which implies the index of $\left[
   \begin{array}{cccc}
     L(\mathcal{A})^{(1)} &  &  &  \\
      & L(\mathcal{A})^{(2)} &  &  \\
      &  & \ddots &  \\
      &  &  & L(\mathcal{A})^{(n_3)} \\
   \end{array}
 \right]$ is $\max\limits_{1\leq i\leq n_3}\{\ind(L(\mathcal{A})^{(i)}\}$. Therefore, $\ind(\matt(\mathcal{A}))=\ind(\mathcal{A})=\max\limits_{1\leq i\leq n_3}\{\ind(L(\mathcal{A})^{(i)})\}$.
\end{proof}

Next, we will give the definition of the Drazin inverse of a tensor. Before that we note that $\mathcal{A}^{k}=\underbrace{\mathcal{A}*_c\cdots*_c\mathcal{A}}_k$.

\begin{definition} \label{d52}
Let $\ac\in \CC^{n_1\times n_1\times n_3}$ and $\ind(\ac)=k$. Then, the tensor $\mathcal{X}\in \CC^{n_1\times n_1\times n_3}$ satisfying
\begin{equation}\label{mpk3}
\mathcal{A}^{k+1}*_c\mathcal{X} = \mathcal{A}^{k}, \quad \mathcal{X}*_c\mathcal{A}*_c\mathcal{X} = \mathcal{X}, \quad \mathcal{A}*_c\mathcal{X} = \mathcal{X}*_c\mathcal{A},
\end{equation}
is called the {\bf{Drazin inverse}} of the
tensor $\ac$ and is denoted by $\ac^D$. Especially, when $k=1$, $\mathcal{X}$ is called the {\bf{group inverse}} of the
tensor $\ac$ and is denoted by $\ac^\#$.
\end{definition}

\begin{lemma}\label{l91}
Let $\ac \in
\CC^{n_1\times n_1\times n_3}$ and
\begin{eqnarray*} 
\DCT(\matt(\mathcal{A}))=\left[
   \begin{array}{cccc}
     L(\mathcal{A})^{(1)} &  &  &  \\
      & L(\mathcal{A})^{(2)} &  &  \\
      &  & \ddots &  \\
      &  &  & L(\mathcal{A})^{(n_3)} \\
   \end{array}
 \right].
\end{eqnarray*}
If $\ind(\mathcal{A})=k$, then the Drazin inverse of $\mathcal{A}$  exists and is unique.
\end{lemma}
\begin{proof}
Since $\ind(\mathcal{A}) = k$, one has that the matrices $L(\mathcal{A})^{(1)}, \ldots, L(\mathcal{A})^{(n_3)}$ are Drazin
invertible. Let $\mathbf{X}_i = (L(\mathcal{A})^{(i)})^D$, $i=1,2,..., n_3$. Then,
\begin{equation*}
\mathcal{X}=\ten(\IDCT(\left[
   \begin{array}{ccc}
    \mathbf{X}_1 &  &    \\
      &   \ddots &  \\
       &  & \mathbf{X}_{n_3} \\
   \end{array}
 \right]))
\end{equation*}
satisfies the three equations of (\ref{mpk3}). It is trivial to see that $\xc$ is the Drazin inverse of $\ac$.

Suppose both tensors $\mathcal{X}$ and $\mathcal{Y}$ are the solutions of (\ref{mpk3}). Let
\begin{eqnarray*} 
\DCT(\matt(\mathcal{X}))=\left[
   \begin{array}{cccc}
     L(\mathcal{X})^{(1)} &  &  &  \\
      & L(\mathcal{X})^{(2)} &  &  \\
      &  & \ddots &  \\
      &  &  & L(\mathcal{X})^{(n_3)} \\
   \end{array}
 \right]
 \end{eqnarray*}
and
\begin{eqnarray*}
\DCT(\matt(\mathcal{A}))=\left[
   \begin{array}{cccc}
     L(\mathcal{Y})^{(1)} &  &  &  \\
      & L(\mathcal{Y})^{(2)} &  &  \\
      &  & \ddots &  \\
      &  &  & L(\mathcal{Y})^{(n_3)} \\
   \end{array}
 \right].
\end{eqnarray*}
It follows $L(\mathcal{X})^{(i)}= (L(\mathcal{A})^{(i)})^D$ and $L(\mathcal{Y})^{(i)}= (L(\mathcal{A})^{(i)})^D$, $i=1,2,...,n_3$. Therefore, $\mathcal{X}$ and $\mathcal{Y}$ coincides since $L(\mathcal{X})^{(i)}$ and $L(\mathcal{Y})^{(i)}$ are the same.
\end{proof}

\begin{theorem}\label{1t1}
Let $\ac \in \CC^{n_1\times n_1\times n_3}$ and $\ind(\mathcal{A})=k$. Then,
\begin{equation*}
\ac^D=\ac^k*_c(\ac^{2k+1})^{(1)}*_c\ac^k.
\end{equation*}
In particular,
\begin{equation*}
\ac^D=\ac^k*_c(\ac^{2k+1})^{\dag}*_c\ac^k.
\end{equation*}
\end{theorem}

\begin{proof}
By the definition of the Drazin inverse, one has
\begin{eqnarray*}
  \mathcal{A}^k &=& \mathcal{A}^{k+1}*_c\mathcal{A}^D=\mathcal{A}^{k+2}*_c(\mathcal{A}^D)^2 \\
   &=& \cdots\\
   &=& \mathcal{A}^{2k}*_c(\mathcal{A}^D)^k\\
   &=& \mathcal{A}^{2k+1}*_c(\mathcal{A}^D)^{k+1}.
\end{eqnarray*}
Let $\xc=\ac^k*_c(\ac^{2k+1})^{(1)}*_c\ac^k$. Therefore, we have
\begin{eqnarray*}
  \mathcal{A}^{k+1}*_c\mathcal{X}&=&\mathcal{A}^{k+1}*_c\ac^k*_c(\ac^{2k+1})^{(1)}*_c\ac^k
  =\mathcal{A}^{2k+1}*_c(\ac^{2k+1})^{(1)}*_c\mathcal{A}^{2k+1}*_c(\mathcal{A}^D)^{k+1}  \\
   &=& \mathcal{A}^{2k+1}*_c(\mathcal{A}^D)^{k+1}=\mathcal{A}^{k},
\end{eqnarray*}
\begin{eqnarray*}
  \mathcal{X}*_c\mathcal{A}*_c\mathcal{X}&=&\ac^k*_c(\ac^{2k+1})^{(1)}*_c\ac^k*_c\ac*_c
  \ac^k*_c(\ac^{2k+1})^{(1)}*_c\ac^k  \\
   &=& \ac^k*_c(\ac^{2k+1})^{(1)}*_c\ac^k=\mathcal{X}.
   \end{eqnarray*}
Moreover,
\begin{eqnarray*}
  \mathcal{A}*_c\mathcal{X}&=&\ac*_c\ac^k*_c(\ac^{2k+1})^{(1)}*_c\ac^k
  =\ac*_c\mathcal{A}^{2k}*_c(\mathcal{A}^D)^k*_c(\ac^{2k+1})^{(1)}*_c
  \mathcal{A}^{2k+1}*_c(\mathcal{A}^D)^{k+1}  \\
   &=& (\mathcal{A}^D)^k*_c\mathcal{A}^{2k+1}*_c(\ac^{2k+1})^{(1)}
   *_c\mathcal{A}^{2k+1}*_c(\mathcal{A}^D)^{k+1}=
   (\mathcal{A}^D)^k*_c\mathcal{A}^{2k+1}*_c(\mathcal{A}^D)^{k+1},
\end{eqnarray*}
and
\begin{eqnarray*}
\mathcal{X}*_c\mathcal{A} &=&\mathcal{A}^{k}*_c(\ac^{2k+1})^{(1)}*_c\mathcal{A}^{k+1}
=\mathcal{A}^{2k+1}*_c(\mathcal{A}^D)^{k+1}*_c(\ac^{2k+1})^{(1)}*_c
  \ac*_c\mathcal{A}^{2k}*_c(\mathcal{A}^D)^k\\
  &=&(\mathcal{A}^D)^{k+1}*_c\mathcal{A}^{2k+1}*_c(\ac^{2k+1})^{(1)}
   *_c\mathcal{A}^{2k+1}*_c(\mathcal{A}^D)^k
   =(\mathcal{A}^D)^{k+1}*_c\mathcal{A}^{2k+1}*_c(\mathcal{A}^D)^k\\
   &=&(\mathcal{A}^D)^k*_c\mathcal{A}^{2k+1}*_c(\mathcal{A}^D)^{k+1},
\end{eqnarray*}
which implies $\mathcal{A}*_c\mathcal{X}=\mathcal{X}*_c\mathcal{A}$. Thus, we obtain
$\ac^D=\ac^k*_c(\ac^{2k+1})^{(1)}*_c\ac^k$. By taking $(\ac^{2k+1})^{\dag}$ for $(\ac^{2k+1})^{(1)}$,
we have
$\ac^D=\ac^k*_c(\ac^{2k+1})^{\dag}*_c\ac^k$.
\end{proof}

\begin{theorem}
Let $\mathcal{A}\in\CC^{n_1\times n_1\times n_3}$ and $\ind(\mathcal{A})=k$. Suppose $\mathcal{A}^k$ has the C-QDR decomposition $\mathcal{A}^k=\mathcal{Q}*_c\mathcal{D}*_c\mathcal{R}$.
Then,
\begin{equation*}
  \mathcal{A}^D=\mathcal{Q}*_c(\mathcal{R}*_c\mathcal{A}*_c\mathcal{Q})^{-1}*_c\mathcal{R}.
\end{equation*}
\end{theorem}
\begin{proof}
Let
\begin{eqnarray*} 
\DCT(\matt(\mathcal{A}))=\left[
   \begin{array}{cccc}
     L(\mathcal{A})^{(1)} &  &  &  \\
      & L(\mathcal{A})^{(2)} &  &  \\
      &  & \ddots &  \\
      &  &  & L(\mathcal{A})^{(n_3)} \\
   \end{array}
 \right].
\end{eqnarray*}
Since $\mathcal{A}^k=\mathcal{Q}*_c\mathcal{D}*_c\mathcal{R}$ is the C-QDR decomposition, we conclude that $(L(\mathcal{A})^{(i)})^k=\mathbf{Q}_i\mathbf{D}_i\mathbf{R}_i$, $\mathbf{Q}_i\in\CC^{n_1\times r}_r$, $\mathbf{D}_i\in\CC^{r\times r}_r$,  $\mathbf{R}_i\in\CC^{r\times n_2}_r$, $i=1,2,...,n_3$ are the QDR decomposition of $(L(\mathcal{A})^{(i)})^k$. Notice that $(L(\mathcal{A})^{(i)})^k=(\mathbf{Q}_i\mathbf{D}_i)\mathbf{R}_i
=\mathbf{Q}_i(\mathbf{D}_i\mathbf{R}_i)$ are full rank decomposition of $(L(\mathcal{A})^{(i)})^k$. By \cite[Theorem 2.1]{SD}, we have $\mathbf{R}_iL(\mathcal{A})^{(i)}\mathbf{Q}_i\mathbf{D}_i$ and
$\mathbf{D}_i\mathbf{R}_iL(\mathcal{A})^{(i)}\mathbf{Q}_i$, $i=1,2,...,n_3$ are invertible. So are
$\mathcal{R}*_c\mathcal{A}*_c\mathcal{Q}*_c\mathcal{D}$ and $\mathcal{D}*_c\mathcal{R}*_c\mathcal{A}*_c\mathcal{Q}$. Hence, we have $\mathcal{R}*_c\mathcal{A}*_c\mathcal{Q}$ is invertible.

On the other hand, by \cite{C}, we conclude that $$(\mathbf{Q}_i\mathbf{D}_i\mathbf{R}_iL(\mathcal{A})^{(i)}\mathbf{Q}_i\mathbf{D}_i\mathbf{R}_i)^\dag
=(\mathbf{D}_i\mathbf{R}_i)^\dag
(\mathbf{R}_iL(\mathcal{A})^{(i)}\mathbf{Q}_i)^{-1}(\mathbf{Q}_i\mathbf{D}_i)^\dag, \ \ i=1,2,...,n_3$$
due to $\mathbf{R}_iL(\mathcal{A})^{(i)}\mathbf{Q}_i$ are invertible, $\mathbf{D}_i\mathbf{R}_i$ are  full row rank and $\mathbf{Q}_i\mathbf{D}_i$ are full column rank. Therefore,
\begin{equation*}
(\mathcal{Q}*_c\mathcal{D}*_c\mathcal{R}*_c\ac*_c\mathcal{Q}*_c\mathcal{D}*_c\mathcal{R})^{\dag}=
(\mathcal{D}*_c\mathcal{R})^{\dag}*_c(\mathcal{R}*_c\ac*_c\mathcal{Q})^{-1}*_c(\mathcal{Q}*_c\mathcal{D})^\dag.
\end{equation*}
By Theorem \ref{1t1}, we have
\begin{eqnarray*}
  \ac^D=\ac^k*_c(\ac^{2k+1})^{\dag}*_c\ac^k &=& \ac^k*_c(\ac^{k}*_c\ac*_c\ac^{k})^{\dag}*_c\ac^k \\
   &=& \mathcal{Q}*_c\mathcal{D}*_c\mathcal{R}*_c (\mathcal{Q}*_c\mathcal{D}*_c\mathcal{R}*_c\ac*_c\mathcal{Q}*_c\mathcal{D}*_c\mathcal{R})^{\dag}
   *_c\mathcal{Q}*_c\mathcal{D}*_c\mathcal{R}\\
&=&  \mathcal{Q}*_c\mathcal{D}*_c\mathcal{R}*_c
(\mathcal{D}*_c\mathcal{R})^{\dag}*_c(\mathcal{R}*_c\ac*_c\mathcal{Q})^{-1}*_c(\mathcal{Q}*_c\mathcal{D})^\dag
*_c\mathcal{Q}*_c\mathcal{D}*_c\mathcal{R}\\
&=&\mathcal{Q}*_c(\mathcal{R}*_c\mathcal{A}*_c\mathcal{Q})^{-1}*_c\mathcal{R}.
\end{eqnarray*}
\end{proof}

In the following, we will establish another expression of the Drazin inverse  by using the core-nilpotent decomposition of the tensors.

\begin{definition}
Let $\ac \in \CC^{n_1\times n_1\times n_3}$. Then,
$$\mathcal{C}_{\mathcal{A}} = \mathcal{A}^2
*_c\mathcal{A}^D
$$
is called the {\bf{core part}} of the tensor $\mathcal{A}$.
\end{definition}


\begin{lemma}
Let $\ac \in \CC^{n_1\times n_1\times n_3}$, $\ind(\mathcal{A})=k$ and $\mathcal{C}_{\mathcal{A}}\in \CC^{n_1\times n_1\times n_3}$ is the core part of the tensor $\mathcal{A}$. Define $\mathcal{N}_{\mathcal{A}}=\mathcal{A}-\mathcal{C}_{\mathcal{A}}$.
Then,
$$\mathcal{N}_{\mathcal{A}}^k = \mathcal{O}\ \ \text{and} \ \ \ind(\mathcal{N}_{\mathcal{A}})=k.$$
\end{lemma}
\begin{proof}
When $\ind(\mathcal{A})=0$, we have $\ac$ is  invertible. Then, $\mathcal{N}_{\mathcal{A}}=\mathcal{O}$ and $\ind(\mathcal{N}_{\mathcal{A}})=0$.

When $\ind(\mathcal{A})\geq1$,
\begin{equation*}
  \mathcal{N}_{\mathcal{A}}^k =(\mathcal{A}-\mathcal{A}^2
*_c\mathcal{A}^D)^k=\mathcal{A}^k*_c(\mathcal{I}-\mathcal{A}*_c\mathcal{A}^D)^k
=\mathcal{A}^k*_c(\mathcal{I}-\mathcal{A}*_c\mathcal{A}^D)=\mathcal{A}^k-\mathcal{A}^k=\mathcal{O}.
\end{equation*}
On the other hand, $\mathcal{N}_{\mathcal{A}}^l=\mathcal{A}^l-\mathcal{A}^{l+1}*_c\mathcal{A}^D\neq\mathcal{O}$
for $l<k$. Hence, we have $\ind(\mathcal{N}_{\mathcal{A}})=k$.
\end{proof}

The $\mathcal{N}_{\mathcal{A}}$ we defined is call the {\bf{nilpotent part}} of the tensor $\mathcal{A}$.

\begin{definition}
Let $\ac \in \CC^{n_1\times n_1\times n_3}$, $\mathcal{C}_{\mathcal{A}}$ be the core part of $\mathcal{A}$ and $\mathcal{N}_{\mathcal{A}}=\mathcal{A}-\mathcal{C}_{\mathcal{A}}$. Then,
\begin{equation*}
  \ac=\mathcal{C}_{\mathcal{A}}+\mathcal{N}_{\mathcal{A}}
\end{equation*}
is called the {\bf{core-nilpotent}} decomposition of the tensor $\mathcal{A}$.
\end{definition}

\begin{theorem}{\rm \cite{WWQ}}\label{taa} Let $\mathbf{A}\in{\CC^{n_1\times n_1}}$, $\ind(\mathbf{A})=k$, and $\mathbf{A}=\mathbf{C_A}+\mathbf{N_A}$  is the core-nilpotent decomposition of $\mathbf{A}$. Then, there exists an invertible matrices $\mathbf{P}\in{\CC^{n_1\times n_1}}$ such that
\begin{eqnarray*}
\mathbf{A}=\mathbf{P}\begin{bmatrix}
  \mathbf{C} & \mathbf{O}   \\
  \mathbf{O} & \mathbf{N}   \\
\end{bmatrix}\mathbf{P}^{-1},
\end{eqnarray*} where $\mathbf{C_A}=\mathbf{P}\begin{bmatrix}
  \mathbf{C} & \mathbf{O}   \\
  \mathbf{O} & \mathbf{O}   \\
\end{bmatrix}\mathbf{P}^{-1}$, $\mathbf{N_A}=\mathbf{P}\begin{bmatrix}
  \mathbf{O} & \mathbf{O}   \\
  \mathbf{O} & \mathbf{N}   \\
\end{bmatrix}\mathbf{P}^{-1}$, $\mathbf{C}\in{\CC^{r\times r}}$, $\mathbf{N}\in{\CC^{(n_1-r)\times (n_1-r)}}$. Besides,
\begin{eqnarray*}
\mathbf{A}^D=\mathbf{P}\begin{bmatrix}
  \mathbf{C}^{-1} & \mathbf{O}   \\
  \mathbf{O} & \mathbf{O}   \\
\end{bmatrix}\mathbf{P}^{-1}.
\end{eqnarray*}.

\end{theorem}

\begin{theorem}
Let $\ac \in \CC^{n_1\times n_1\times n_3}$ and $\ind(\mathcal{A})=k$.
Then
\begin{equation} \label{a}
\ac = \mathcal{P}*_c\Phi*_c\mathcal{P}^{-1},
\end{equation}
where $\mathcal{P} \in \CC^{n_1\times n_1\times n_3}$ is  an invertible tensor,
\begin{equation*}
  \Phi=\ten(\IDCT(
\begin{bmatrix}
\mat{\mathbf{C}_1}{\mathbf{O}}{\mathbf{O}}{\mathbf{N}_1} &  & \\ & \ddots & \\ &  & \mat{\mathbf{C}_{n_3}}{\mathbf{O}}{\mathbf{O}}{\mathbf{N}_{n_3}}
\end{bmatrix})),  \  \ \mat{\mathbf{C}_i}{\mathbf{O}}{\mathbf{O}}{\mathbf{N}_i}= \mathbf{P}^{-1}_i(\mathbf{C}_{\mathbf{A}_i}+\mathbf{N}_{\mathbf{A}_i})\mathbf{P}_i,
\end{equation*}
$\mathbf{C}_{\mathbf{A}_i}$ and $\mathbf{N}_{\mathbf{A}_i}$ are the core and nilpotent part of $L(\mathcal{A})^{(i)}$, $i=1,2,...,n_3$, respectively. Furthermore, if $rank(\mathbf{C}_i)=r$, $i=1,2,...,n_3$, then
\begin{equation*} \label{a}
\ac = \mathcal{P}*_c\mat{\mathcal{C}}{\mathcal{O}}{\mathcal{O}}{\mathcal{N}}*_c\mathcal{P}^{-1}.
\end{equation*}
Besides,
\begin{equation*} \label{a}
\ac^D = \mathcal{P} *_c\mat{\mathcal{C}^{-1}}{\mathcal{O}}{\mathcal{O}}{\mathcal{O}}*_c\mathcal{P}^{-1}.
\end{equation*}
\end{theorem}

\begin{proof}
Suppose \begin{eqnarray*} 
\DCT(\matt(\mathcal{A}))=\left[
   \begin{array}{cccc}
     L(\mathcal{A})^{(1)} &  &  &  \\
      & L(\mathcal{A})^{(2)} &  &  \\
      &  & \ddots &  \\
      &  &  & L(\mathcal{A})^{(n_3)} \\
   \end{array}
 \right].
\end{eqnarray*}
Then, by using Theorem \ref{taa}, we have
\begin{eqnarray*}
&&\begin{bmatrix} L(\mathcal{A})^{(1)} & & \\
 & \ddots & \\
 & & L(\mathcal{A})^{(n_3)} \end{bmatrix}\\
 &&=\begin{bmatrix}\mathbf{P}_1\begin{bmatrix}
  \mathbf{C}_1 & \mathbf{O}   \\
  \mathbf{O} & \mathbf{N}_1   \\
\end{bmatrix}\mathbf{P}_1^{-1} & & \\
 & \ddots & \\
 & & \mathbf{P}_{n_3}\begin{bmatrix}
  \mathbf{C}_{n_3} & \mathbf{O}   \\
  \mathbf{O} & \mathbf{N}_{n_3}   \\
\end{bmatrix}\mathbf{P}_{n_3}^{-1} \end{bmatrix}\\
&&=\begin{bmatrix}
                     \mathbf{P}_1 &   &   \\
                       & \ddots   &   \\
                       &   &  \mathbf{P}_{n_3} \\
                   \end{bmatrix}\begin{bmatrix}
                     \begin{bmatrix}
  \mathbf{C}_1 & \mathbf{O}   \\
  \mathbf{O} & \mathbf{N}_1   \\
\end{bmatrix} &   &   \\
                       & \ddots   &   \\
                       &   &  \begin{bmatrix}
  \mathbf{C}_{n_3} & \mathbf{O}   \\
  \mathbf{O} & \mathbf{N}_{n_3}   \\
\end{bmatrix} \\
                   \end{bmatrix}\begin{bmatrix}
                     \mathbf{P}^{-1}_1 &   &   \\
                       & \ddots   &   \\
                       &   &  \mathbf{P}^{-1}_{n_3} \\
                   \end{bmatrix}.
\end{eqnarray*}
Executing $\ten(\IDCT)(\cdot)$ on the tensors of the both sides of the equation, we have
$$\mathcal{A}=\mathcal{P}*_c\Phi*_c\mathcal{P}^{-1},$$
where \begin{equation*}
  \Phi=\ten(\IDCT(
\begin{bmatrix}
\mat{\mathbf{C}_1}{\mathbf{O}}{\mathbf{O}}{\mathbf{N}_1} &  & \\ & \ddots & \\ &  & \mat{\mathbf{C}_{n_3}}{\mathbf{O}}{\mathbf{O}}{\mathbf{N}_{n_3}}
\end{bmatrix})).
\end{equation*}
Again by using Theorem \ref{taa}, we have
\begin{eqnarray*}
&&\begin{bmatrix}
(L(\mathcal{A})^{(1)})^D & & \\
 & \ddots & \\
 & & (L(\mathcal{A})^{(n_3)})^D \end{bmatrix}\\
 &&=\begin{bmatrix}(\mathbf{P}_1\begin{bmatrix}
  \mathbf{C}_1 & \mathbf{O}   \\
  \mathbf{O} & \mathbf{N}_1   \\
\end{bmatrix}\mathbf{P}_1^{-1})^D & & \\
 & \ddots & \\
 & & (\mathbf{P}_{n_3}\begin{bmatrix}
  \mathbf{C}_{n_3} & \mathbf{O}   \\
  \mathbf{O} & \mathbf{N}_{n_3}   \\
\end{bmatrix}\mathbf{P}_{n_3}^{-1})^D \end{bmatrix}\\
&&=\begin{bmatrix}
                     \mathbf{P}_1 &   &   \\
                     & \ddots   &   \\
                     &   &  \mathbf{P}_{n_3} \\
                   \end{bmatrix}\begin{bmatrix}
                     \begin{bmatrix}
  \mathbf{C}_1 & \mathbf{O}   \\
  \mathbf{O} & \mathbf{N}_1   \\
\end{bmatrix}^D &   &   \\
                       & \ddots   &   \\
                       &   &  \begin{bmatrix}
  \mathbf{C}_{n_3} & \mathbf{O}   \\
  \mathbf{O} & \mathbf{N}_{n_3}   \\
\end{bmatrix}^D \\
                   \end{bmatrix}\begin{bmatrix}
                    \mathbf{P}^{-1}_1 &   &   \\
                       & \ddots   &   \\
                       &   &  \mathbf{P}^{-1}_{n_3} \\
                   \end{bmatrix}\\
&&=\begin{bmatrix}
                     \mathbf{P}_1 &   &   \\
                       & \ddots   &   \\
                       &   &  \mathbf{P}_{n_3} \\
                   \end{bmatrix}\begin{bmatrix}
                     \begin{bmatrix}
  \mathbf{C}^{-1}_1 & \mathbf{O}   \\
  \mathbf{O} & \mathbf{O}   \\
\end{bmatrix} &   &   \\
                       & \ddots   &   \\
                       &   &  \begin{bmatrix}
  \mathbf{C}^{-1}_{n_3} & \mathbf{O}   \\
  \mathbf{O} & \mathbf{O}   \\
\end{bmatrix} \\
                   \end{bmatrix}\begin{bmatrix}
                     \mathbf{P}^{-1}_1 &   &   \\
                       & \ddots   &   \\
                       &   &  \mathbf{P}^{-1}_{n_3} \\
                   \end{bmatrix}.
\end{eqnarray*}
Executing $\ten(\IDCT)(\cdot)$ on the tensors of the both sides of the equation, we have
\begin{equation*} \label{a}
\ac^D = \mathcal{P} *_c\Phi^D*_c\mathcal{P}^{-1},
\end{equation*}
where
\begin{equation*}
  \Phi^D=\ten(\IDCT(
\begin{bmatrix}
\mat{\mathbf{C}^{-1}_1}{\mathbf{O}}{\mathbf{O}}{\mathbf{O}} &  & \\ & \ddots & \\ &  & \mat{\mathbf{C}^{-1}_{n_3}}{\mathbf{O}}{\mathbf{O}}{\mathbf{O}}
\end{bmatrix})).
\end{equation*}
When $rank(\mathbf{C}_1)=rank(\mathbf{C}_2)=\cdots=rank(\mathbf{C}_{n_3})=r$, one has
\begin{eqnarray*}
  \mat{\mathcal{C}}{\mathcal{O}}{\mathcal{O}}{\mathcal{N}}=\ten(\IDCT(
\begin{bmatrix}
\mat{\mathbf{C}_1}{\mathbf{O}}{\mathbf{O}}{\mathbf{N}_1} &  & \\ & \ddots & \\ &  & \mat{\mathbf{C}_{n_3}}{\mathbf{O}}{\mathbf{O}}{\mathbf{N}_{n_3}}
\end{bmatrix})), \ \text{where} \ \mathcal{C}\in\CC^{r\times r\times n_3}, \   \mathcal{N}\in\CC^{(n_1-r)\times (n_1-r)\times n_3}.
\end{eqnarray*}
Hence,
\begin{eqnarray*} \label{a}
\ac=\mathcal{P}*_c
\mat{\mathcal{C}}{\mathcal{O}}{\mathcal{O}}{\mathcal{N}}
*_c\mathcal{P}^{-1}.
\end{eqnarray*}
Since
\begin{equation*} \label{a}
\ac^D = \mathcal{P} *\ten(\IDCT(
\begin{bmatrix}
\mat{\mathbf{C}^{-1}_1}{\mathbf{O}}{\mathbf{O}}{\mathbf{O}} &  & \\ & \ddots & \\ &  & \mat{\mathbf{C}^{-1}_{n_3}}{\mathbf{O}}{\mathbf{O}}{\mathbf{O}}
\end{bmatrix}))*\mathcal{P}^{-1},
\end{equation*}
it is trivial to see
\begin{equation*}
\ac^D = \mathcal{P} *_c\mat{\mathcal{C}^{-1}}{\mathcal{O}}{\mathcal{O}}{\mathcal{O}}*_c\mathcal{P}^{-1}.
\end{equation*}
\end{proof}

\begin{theorem}
Let $\mathcal{A}\in\CC^{n_1\times n_1\times n_3}$. Suppose $\mathcal{A}$ has the C-HS decomposition.
Then,
\begin{equation*}
\ac^D = \uc *_c
\mat{(\mathcal{{S}}_r*_c\mathcal{{K}})^D}
{((\mathcal{{S}}_r*_c\mathcal{{K}})^D)^2*_c\mathcal{{S}}_r*_c\mathcal{L}}{\mathcal{O}}{\mathcal{O}}*_c\uc^H.
\end{equation*}
\end{theorem}

\begin{proof}
Let
\begin{equation*}
\mathcal{A}=\mathcal{U}*_c\left[
\begin{array}{cc}
\mathcal{S}_r*_c\mathcal{K} & \mathcal{S}_r*_c\mathcal{L} \\
\mathcal{O} & \mathcal{O} \\
\end{array}
\right]*_c\mathcal{U}^H,
\end{equation*}
where $\mathcal{S}_r, \mathcal{K}\in\CC^{r\times r\times n_3}$, $\mathcal{L}\in\CC^{r\times(n_1-r)\times n_3}$, be the C-HS decomposition of $\ac$. Suppose
\begin{equation*}
\xc= \uc *_c
\mat{\xc_1}
{\xc_2}{\xc_3}{\xc_4}*_c\uc^H,
\end{equation*}
where $\xc_1\in\CC^{r\times r\times n_3}, \xc_2\in\CC^{r\times (n_1-r)\times n_3}, \xc_3\in\CC^{(n_1-r)\times r\times n_3},
\xc_4\in\CC^{(n_1-r)\times(n_1-r)\times n_3}$, is the Drazin inverse of $\ac$. Thus, $\xc$ satisfies the three tensor equation in (\ref{mpk3}). Hence, by $\xc*_c\ac*_c\xc=\xc$, we have
\begin{equation*}
\mathcal{X}_1*_c\mathcal{S}_r*_c\mathcal{K}*_c\mathcal{X}_1+
\mathcal{X}_1*_c\mathcal{S}_r*_c\mathcal{L}*_c\mathcal{X}_3=\mathcal{X}_1,
\end{equation*}
\begin{equation*}
\mathcal{X}_1*_c\mathcal{S}_r*_c\mathcal{K}*_c\mathcal{X}_2+
\mathcal{X}_1*_c\mathcal{S}_r*_c\mathcal{L}*_c\mathcal{X}_4=\mathcal{X}_2,
\end{equation*}
\begin{equation*}
\mathcal{X}_3*_c\mathcal{S}_r*_c\mathcal{K}*_c\mathcal{X}_1+
\mathcal{X}_3*_c\mathcal{S}_r*_c\mathcal{L}*_c\mathcal{X}_3=\mathcal{X}_3,
\end{equation*}
\begin{equation*}
\mathcal{X}_3*_c\mathcal{S}_r*_c\mathcal{K}*_c\mathcal{X}_2+
\mathcal{X}_3*_c\mathcal{S}_r*_c\mathcal{L}*_c\mathcal{X}_4=\mathcal{X}_4.
\end{equation*}
By $\ac*_c\xc=\xc*_c\ac$, we have
\begin{equation*}
\mathcal{X}_1*_c\mathcal{S}_r*_c\mathcal{K}=
\mathcal{S}_r*_c\mathcal{K}*_c\mathcal{X}_1+\mathcal{S}_r*_c\mathcal{L}*_c\mathcal{X}_3,
\end{equation*}
\begin{equation*}
\mathcal{X}_1*_c\mathcal{S}_r*_c\mathcal{L}=
\mathcal{S}_r*_c\mathcal{K}*_c\mathcal{X}_2+\mathcal{S}_r*_c\mathcal{L}*_c\mathcal{X}_4,
\end{equation*}
\begin{equation*}
\mathcal{X}_3*_c\mathcal{S}_r*_c\mathcal{K}=\mathcal{O},
\end{equation*}
\begin{equation*}
\mathcal{X}_3*_c\mathcal{S}_r*_c\mathcal{L}=\mathcal{O}.
\end{equation*}
By $\mathcal{A}^{k+1}*_c\xc=\ac^k$, we have
\begin{equation*}
(\mathcal{S}_r*_c\mathcal{K})^{k+1}*_c\mathcal{X}_1+
(\mathcal{S}_r*_c\mathcal{K})^{k}*_c\mathcal{S}_r*_c\mathcal{L}*_c\mathcal{X}_3
=(\mathcal{S}_r*_c\mathcal{K})^{k},
\end{equation*}
\begin{equation*}
(\mathcal{S}_r*_c\mathcal{K})^{k+1}*_c\mathcal{X}_2+
(\mathcal{S}_r*_c\mathcal{K})^{k}*_c\mathcal{S}_r*_c\mathcal{L}*_c\mathcal{X}_4
=(\mathcal{S}_r*_c\mathcal{K})^{k-1}*_c\mathcal{S}_r*_c\mathcal{L}.
\end{equation*}
Thus, $\mathcal{X}_3=0$ and $\mathcal{X}_4=0$. In addition,
\begin{equation*}\label{bb}
  \mathcal{X}_1*_c\mathcal{S}_r*_c\mathcal{K}*_c\mathcal{X}_1=\mathcal{X}_1, \
  \mathcal{X}_1*_c\mathcal{S}_r*_c\mathcal{K}=\mathcal{S}_r*_c\mathcal{K}*_c\mathcal{X}_1, \
(\mathcal{S}_r*_c\mathcal{K})^{k+1}*_c\mathcal{X}_1=(\mathcal{S}_r*_c\mathcal{K})^{k},
\end{equation*}
which implies $\mathcal{X}_1=(\mathcal{S}_r*_c\mathcal{K})^D$. Moreover,
\begin{equation*}\label{bb}
  \mathcal{X}_1*_c\mathcal{S}_r*_c\mathcal{K}*_c\mathcal{X}_2=\mathcal{X}_2, \
  \mathcal{X}_1*_c\mathcal{S}_r*_c\mathcal{L}=\mathcal{S}_r*_c\mathcal{K}*_c\mathcal{X}_2,
\end{equation*}
which implies $\mathcal{X}_2=((\mathcal{S}_r*_c\mathcal{K})^D)^2*_c\mathcal{S}_r*_c\mathcal{L}$. Therefore,
\begin{equation*}
\ac^D = \uc *_c
\mat{(\mathcal{{S}}_r*_c\mathcal{{K}})^D}
{((\mathcal{{S}}_r*_c\mathcal{{K}})^D)^2*_c\mathcal{{S}}_r*_c\mathcal{L}}{\mathcal{O}}{\mathcal{O}}*_c\uc^H.
\end{equation*}
\end{proof}

\subsection{The algorithm for computing the Drazin inverse of a tensor}

In the following, we construct an algorithm to compute the Drazin inverse of a tensor based on Theorem \ref{1t1}.

\begin{algorithm}[H]
\caption{\textsc{Compute the Drazin inverse of a tensor $\mathcal{A}$}}\label{algo102}
\KwIn{$n_1\times n_2\times n_3$ tensor $\mathcal{A}$}
\KwOut{$n_2\times n_1\times n_3$ tensor $\mathcal{X}$}
\begin{enumerate}
\addtolength{\itemsep}{-0.8\parsep minus 0.8\parsep}

\item $\widehat{\mathcal{A}} = L(\mathcal{A})=\mathcal{A}\times_3\mathbf{M}$, where $\mathbf{M}$ is defined in (\ref{l201})

\item $k=\max\limits_{1\leq i\leq n_3}\{\ind(\widehat{\mathcal{A}}^{(i)})\}$

\item $\widehat{\mathcal{B}} = L(\mathcal{A}^k)=\mathcal{A}^k\times_3\mathbf{M}$,
$\widehat{\mathcal{C}} = L(\mathcal{A}^{2k+1})=\mathcal{A}^{2k+1}\times_3\mathbf{M}$

\item  for $i=1,\ldots, n_3$

\qquad  $\widehat{\mathcal{H}}^{(i)}=\pinv(\widehat{\mathcal{C}}^{(i)})$; where
$\pinv(\widehat{\mathcal{C}}^{(i)})$ is the Moore-Penrose inverse of $\widehat{\mathcal{C}}^{(i)}$

end

\item

for $i=1,\ldots, n_3$

\quad $\widehat{\mathcal{X}}^{(i)}=\widehat{\mathcal{B}}^{(i)}
\widehat{\mathcal{H}}^{(i)}\widehat{\mathcal{B}}^{(i)}$

end

\item  ${\mathcal{X}} = L^{-1}(\widehat{\mathcal{X}})=\widehat{\mathcal{X}}\times_3\mathbf{M}^{-1}$
\end{enumerate}
\end{algorithm}

\begin{example}
Let $\ac \in \CC^{3\times 3\times 3}$ with frontal slices
$$\ac^{(1)}=\begin{bmatrix}
    2  & 0 & 0\\
    1  & 3 & 0\\
    0  & 0 & 0\\
  \end{bmatrix},\quad
\ac^{(2)}=\begin{bmatrix}
    1  & 3 & 3\\
    0  & 4 & 5\\
    3  & 0 & 0\\
  \end{bmatrix},\quad
\ac^{(3)}=\begin{bmatrix}
    3  & 2 & 0\\
    0  & 1 & 3\\
    2  & 0 & 1\\
  \end{bmatrix}.$$
Then, by using Algorithm \ref{algo102}, we have
\begin{eqnarray*}
(\ac^D)^{(1)}=\begin{bmatrix}
     0.0007  &  0.0123  & -0.1008 \\
     -0.1030 &  0.0358  & 0.0223\\
     -0.0036 &  -0.0617 &   0.0042 \\
  \end{bmatrix},\quad
(\ac^D)^{(2)}=\begin{bmatrix}
     0.2056  & -0.0473  &  0.6283 \\
     0.0145  &  0.0637 &  -0.1531\\
     0.1721  &  0.0365 &   0.0585 \\
  \end{bmatrix},
  \end{eqnarray*}
\begin{eqnarray*}
(\ac^D)^{(3)}=\begin{bmatrix}
     -0.1937 &  0.0317  & -0.5392 \\
     0.1115  &  -0.1005 &   0.0693 \\
     -0.2316 &  0.0415  &  -0.0040 \\
  \end{bmatrix}.
\end{eqnarray*}
\end{example}





\section{The inverse along a tensor under the C-product}\label{four}

In this section, we firstly define the inverse along a tensor under the C-product and then give some representations of this inverse. Moreover, an algorithm is built to compute the inverse along a tensor.

\subsection{The expressions of the inverse along a tensor}

\begin{definition} 
Let $\ac \in \CC^{n_1\times n_2\times n_3}$ and
$\gc \in \CC^{n_2\times n_1\times n_3}$. If there exist tensors
$\xc \in \CC^{n_2\times n_1\times n_3}$,
$\uc \in \CC^{n_1\times n_1\times n_3}$ and
$\vc \in \CC^{n_2\times n_2\times n_3}$ such that
\begin{equation} \label{matrizmary}
\xc*_c\ac*_c\gc = \gc, \qquad \gc*_c\ac*_c\xc = \gc, \qquad \xc = \gc*_c\uc, \qquad \xc^H = \gc^H*_c\vc,
\end{equation}
then $\xc$ is called the \textbf{inverse along $\gc$} and is denoted by $\mary{\ac}{\gc}$.
\end{definition}

\begin{theorem}
Let $\ac \in \CC^{n_1\times n_2\times n_3}$ and
$\gc \in \CC^{n_2\times n_1\times n_3}$.
If $\ac$ is invertible along $\gc$, then the inverse of $\ac$ along $\gc$ is unique.
\end{theorem}

\begin{proof}
Let $\xc_1,\xc_2 \in \CC^{n_2\times n_1\times n_3}$ be two inverses of $\ac$ along $\gc$.
There exist tensors $\uc_1,\uc_2,\vc_1,\vc_2 $ of adequate size such that
\begin{equation*}
\xc_{i}*_c\ac*_c\gc = \gc, \qquad \gc*_c\ac*_c\xc_{i} = \gc, \qquad \xc_{i} = \gc*_c\uc_{i}, \qquad \xc_{i}^H = \gc^H*_c\vc_{i},
\end{equation*}
for $i=1,2$. Now we have
\begin{equation*}
\xc_1=\gc*_c\uc_1=\xc_2*_c\ac*_c\gc*_c\uc_1=\xc_2*_c\ac*_c\xc_1=\vc_2^H*_c\gc*_c\ac*_c\xc_1=\vc_2^H*_c\gc=\xc_2,
\end{equation*}
The proof is finished.
\end{proof}

\begin{theorem}\label{tt10}
Let $\ac \in \CC^{n_1\times n_2\times n_3}$,
$\gc \in\CC^{n_2\times n_1\times n_3}$. If $\mathcal{A}$ is invertible along $\gc $, then
\begin{equation*}
  \mary{\mathcal{A}}{\mathcal{G}}=\mathcal{G}*_c(\mathcal{G}*_c\mathcal{A}*_c\mathcal{G})^\dag *_c\mathcal{G}.
\end{equation*}
\end{theorem}

\begin{proof}
Suppose \begin{eqnarray*} 
\DCT(\matt(\mathcal{A}))=\left[
   \begin{array}{cccc}
     L(\mathcal{A})^{(1)} &  &  \\
      & \ddots &  \\
      &  & L(\mathcal{A})^{(n_3)} \\
   \end{array}
 \right]
\end{eqnarray*}
and\begin{eqnarray*} 
\DCT(\matt(\mathcal{G}))=\left[
   \begin{array}{cccc}
     L(\mathcal{G})^{(1)} &  &  \\
      & \ddots &  \\
      &  & L(\mathcal{G})^{(n_3)} \\
   \end{array}
 \right].
\end{eqnarray*}
Let $\mathbf{\overline{A}}_{i}=L(\ac)^{(i)}$ and $\mathbf{\overline{G}}_{i}=L(\gc)^{(i)}$. By \cite{bbj}, we have
$$\mary{\mathbf{\overline{A}}_{i}}{\mathbf{\overline{G}}_{i}}=
\mathbf{\overline{G}}_{i}(\mathbf{\overline{G}}_{i}\mathbf{\overline{A}}_{i}
\mathbf{\overline{G}}_{i})^\dag\mathbf{\overline{G}}_{i}, \  i=1,2,...,n_3.$$
Then,
\begin{align*} 
\DCT(\matt(\mary{\ac}{\gc}))
&= \begin{bmatrix}
     \mary{\mathbf{\overline{A}}_{1}}{\mathbf{\overline{G}}_{1}}   &  \\
      & \ddots &  \\
      &  & \mary{\mathbf{\overline{A}}_{n_3}}{\mathbf{\overline{G}}_{n_3}} \\
\end{bmatrix}\\
&=\begin{bmatrix}
     \mathbf{\overline{G}}_{1}(\mathbf{\overline{G}}_{1}
     \mathbf{\overline{A}}_{1}\mathbf{\overline{G}}_{1})^\dag{\mathbf{\overline{G}}_{1}} &  &  \\
      & \ddots &  \\
      &  & \mathbf{\overline{G}}_{n_3}(\mathbf{\overline{G}}_{n_3}\mathbf{\overline{A}}_{n_3}
      \mathbf{\overline{G}}_{n_3})^\dag{\mathbf{\overline{G}}_{n_3}} \\
\end{bmatrix}\\
&=\begin{bmatrix}
                     \mathbf{\overline{G}}_{1} &   &   \\
                       & \ddots   &   \\
                       &   &  \mathbf{\overline{G}}_{n_3} \\
                   \end{bmatrix}
\begin{bmatrix}
                     (\mathbf{\overline{G}}_{1}\mathbf{\overline{A}}_{1}
                     \mathbf{\overline{G}}_{1})^\dag &   &   \\
                       & \ddots   &   \\
                       &   &  (\mathbf{\overline{G}}_{n_3}\mathbf{\overline{A}}_{n_3}
                       \mathbf{\overline{G}}_{n_3})^\dag \\
                   \end{bmatrix}
                   \begin{bmatrix}
                     \mathbf{\overline{G}}_{1} &   &   \\
                       & \ddots   &   \\
                       &   &  \mathbf{\overline{G}}_{n_3} \\
                   \end{bmatrix}.
\end{align*}
Therefore, implementing $\ten(\IDCT)(\cdot)$ on both sides of the equation above, we get
$\mary{\mathcal{A}}{\mathcal{G}}=\mathcal{G}*_c(\mathcal{G}*_c\mathcal{A}*_c\mathcal{G})^\dag *_c\mathcal{G}$.
\end{proof}

\begin{theorem}\label{tt8}
Let $\ac \in \CC^{n_1\times n_2\times n_3}$,
$\gc \in \CC^{n_2\times n_1\times n_3}$. Suppose $\mathcal{G}=\mathcal{M}*_c\mathcal{N}$
is the C-full rank decomposition of $\mathcal{G}$. If $\mathcal{A}$ is invertible along $\mathcal{G}$, then
\begin{equation*}
  \mary{\mathcal{A}}{\mathcal{G}}=\mathcal{M}*_c
  (\mathcal{N}*_c\mathcal{A}*_c\mathcal{M})^{-1}*_c\mathcal{N}.
\end{equation*}
\end{theorem}
\begin{proof}
Let \begin{eqnarray*} 
\DCT(\matt(\mathcal{A}))=\left[
   \begin{array}{cccc}
     L(\mathcal{A})^{(1)} &  &  \\
      & \ddots &  \\
      &  & L(\mathcal{A})^{(n_3)} \\
   \end{array}
 \right]
\end{eqnarray*}
and\begin{eqnarray*} 
\DCT(\matt(\mathcal{G}))=\left[
   \begin{array}{cccc}
     L(\mathcal{G})^{(1)} &  &  \\
      & \ddots &  \\
      &  & L(\mathcal{G})^{(n_3)} \\
   \end{array}
 \right].
\end{eqnarray*}
On the other hand,
$$\DCT(\matt(\mathcal{M}*_c\mathcal{N}))=\begin{bmatrix}
             L(\mathcal{M})^{(1)}L(\mathcal{N})^{(1)}  &  &  \\
             &  \ddots &  \\
             &  &  L(\mathcal{M})^{(n_3)}L(\mathcal{N})^{(n_3)}  \\
             \end{bmatrix}.$$
Let $\mathbf{\overline{A}}_i = L(\mathcal{A})^{(i)}$, $\mathbf{\overline{G}}_i = L(\mathcal{G})^{(i)}$, $\mathbf{\overline{M}}_i = L(\mathcal{M})^{(i)}$, $\mathbf{\overline{N}}_i = L(\mathcal{N})^{(i)}$, $i=1,2,..., n_3$.
Thus, we have $\mathbf{\overline{G}}_i=\mathbf{\overline{M}}_i\mathbf{\overline{N}}_i, i=1,2,...n_3$,
which are the full rank decomposition of $\mathbf{\overline{G}}_i$.
By \cite{bbj}, we have
$$\mary{\mathbf{\overline{A}}_i}{\mathbf{\overline{G}}_i}
={\mathbf{\overline{M}}_i}({\mathbf{\overline{N}}_i}
      {\mathbf{\overline{A}}_i}{\mathbf{\overline{M}}_i})^{-1}{\mathbf{\overline{N}}_i}, \
i=1,2,...,n_3.$$
Therefore,
\begin{align*} 
\DCT(\matt(\mary{\mathcal{A}}{\mathcal{G}})
&= \begin{bmatrix}
     \mary{\mathbf{\overline{A}}_1}{\mathbf{\overline{G}}_1} &  &  \\
      & \ddots &  \\
      &  & \mary{\mathbf{\overline{A}}_{n_3}}{\mathbf{\overline{G}}_{n_3}} \\
\end{bmatrix}\\
&=\begin{bmatrix}
      {\mathbf{\overline{M}}_1}({\mathbf{\overline{N}}_1}
      {\mathbf{\overline{A}}_1}{\mathbf{\overline{M}}_1})^{-1}{\mathbf{\overline{N}}_1} &  &  \\
      & \ddots &  \\
      &  & {\mathbf{\overline{M}}_{n_3}}({\mathbf{\overline{N}}_{n_3}}
      {\mathbf{\overline{A}}_{n_3}}{\mathbf{\overline{M}}_{n_3}})^{-1}{\mathbf{\overline{N}}_{n_3}} \\
\end{bmatrix}\\
&=\begin{bmatrix}
                     {\mathbf{\overline{M}}_1} &   &   \\
                       & \ddots   &   \\
                       &   &  {\mathbf{\overline{M}}_{n_3}} \\
                   \end{bmatrix}
                   \begin{bmatrix}
                     ({\mathbf{\overline{N}}_1}{\mathbf{\overline{A}}_1}
                     {\mathbf{\overline{M}}_1})^{-1} &   &   \\
                       & \ddots   &   \\
                       &   &  ({\overline{N}_{n_3}}{\mathbf{\overline{A}}_{n_3}}
                       {\mathbf{\overline{M}}_{n_3}})^{-1} \\
                   \end{bmatrix}
                   \begin{bmatrix}
                     {\mathbf{\overline{N}}_1} &   &   \\
                       & \ddots   &   \\
                       &   &  {\mathbf{\overline{N}}_{n_3}} \\
                   \end{bmatrix}.
\end{align*}
Performing $\ten(\IDCT)(\cdot)$ on both sides of the equation above, one has
$\mary{\mathcal{A}}{\mathcal{G}}=\mathcal{M}*_c
(\mathcal{N}*_c\mathcal{A}*_c\mathcal{M})^{-1}*_c\mathcal{N}.$
\end{proof}

\begin{theorem}\label{15t}
Let $\ac \in \CC^{n_1\times n_2\times n_3}$,
$\gc \in\CC^{n_2\times n_1\times n_3}$ and
\begin{equation} \label{d}
\gc = \uc*_c\scc*_c\vc^H,
\end{equation}
be the C-SVD of $\gc$. Suppose that $rank\left(L(\scc^{(i)})\right)=r_i$, $i=1,2,...,n_3$. If $\ac$ is represented as
\begin{equation} \label{a}
\ac = \vc *_c \ten(\IDCT(
\begin{bmatrix}
\mat{\mathbf{X}_1}{\bigstar}{\bigstar}{\bigstar} &  & \\ & \ddots & \\ &  & \mat{\mathbf{X}_{n_3}}{\bigstar}{\bigstar}{\bigstar}
\end{bmatrix}))*_c\uc^H,
\end{equation}
where $\mathbf{X}_i \in \CC^{r_i\times r_i}$, $i=1,2,\ldots,n_3$,
then $\mary{\ac}{\gc}$ exists if and only if $\mathbf{X}_i$, $i=1,2,...,n_3$, are nonsingular. In particular, if $rank\left(L(\scc^{(i)})\right)=r$, $i=1,2,...,n_3$, then
\begin{equation*} 
\mary{\ac}{\gc} = \uc*_c\mat{\mathcal{X}^{-1}}{\mathcal{O}}{\mathcal{O}}{\mathcal{O}}
*_c\mathcal{V}^H.
\end{equation*}
\end{theorem}

\begin{proof}
Let \begin{eqnarray*} 
\DCT(\matt(\mathcal{A}))=\left[
   \begin{array}{cccc}
     L(\mathcal{A})^{(1)} &  &  \\
      & \ddots &  \\
      &  & L(\mathcal{A})^{(n_3)} \\
   \end{array}
 \right]
\end{eqnarray*}
and\begin{eqnarray*} 
\DCT(\matt(\mathcal{G}))=\left[
   \begin{array}{cccc}
     L(\mathcal{G})^{(1)} &  &  \\
      & \ddots &  \\
      &  & L(\mathcal{G})^{(n_3)} \\
   \end{array}
 \right].
\end{eqnarray*}
Denote $\mathbf{\overline{A}}_i = L(\mathcal{A})^{(i)}$ and $\mathbf{\overline{G}}_i = L(\mathcal{G})^{(i)}$. Thus,
\begin{align*}
\DCT(\matt(\mary{\mathcal{A}}{\mathcal{G}})
&=\begin{bmatrix}
     \mary{\mathbf{\overline{A}}_1}{\mathbf{\overline{G}}_1} &  &  \\
      & \ddots &  \\
      &  & \mary{\mathbf{\overline{A}}_{n_3}}{\mathbf{\overline{G}}_{n_3}} \\
\end{bmatrix}.
\end{align*}
So, $\mary{\ac}{\gc}$ exists if and only if
$\mary{\mathbf{\overline{A}}_{i}}{\mathbf{\overline{G}}_{i}}$ exists, $i=1,2,\ldots,n_3$. Since
$\gc=\uc*_c\scc*_c\vc^H$ is the C-SVD of $\gc$, we have
\begin{eqnarray*} 
\DCT(\matt(\mathcal{G}))&=&
\left[
   \begin{array}{cccc}
     L(\mathcal{G})^{(1)} &  &  \\
      & \ddots &  \\
      &  & L(\mathcal{G})^{(n_3)}\\
   \end{array}
 \right]\\
 &=&
\left[
   \begin{array}{cccc}
     \mathbf{\overline{G}}_1 &  &  \\
      & \ddots &  \\
      &  & \mathbf{\overline{G}}_{n_3}\\
   \end{array}
 \right]\\
&=&\DCT(\matt(\mathcal{U}))\DCT(\matt(\mathcal{S}))\DCT(\matt(\mathcal{V}{^H})) \\
&=&\left[
   \begin{array}{cccc}
     L(\mathcal{U})^{(1)} &  &  \\
      & \ddots &  \\
      &  & L(\mathcal{U})^{(n_3)}\\
   \end{array}
 \right]\left[
   \begin{array}{cccc}
     L(\mathcal{S})^{(1)} &  &  \\
      & \ddots &  \\
      &  & L(\mathcal{S})^{(n_3)}\\
   \end{array}
 \right]\left[
   \begin{array}{cccc}
     L(\mathcal{V})^{(1)} &  &  \\
      & \ddots &  \\
      &  & L(\mathcal{V})^{(n_3)}\\
   \end{array}
 \right]^H.
\end{eqnarray*}
Let $\mathbf{\overline{U}}_i = L(\mathcal{U})^{(i)}$, $\mathbf{\overline{S}}_i = L(\mathcal{S})^{(i)}$ and $\mathbf{\overline{V}}_i = L(\mathcal{V})^{(i)}$. Hence,
$$
{\mathbf{\overline{G}}_{i}}={\mathbf{\overline{U}}_{i}}\begin{bmatrix}
{\mathbf{\overline{S}}_{i}} & \mathbf{O}\\
\mathbf{O}  &  \mathbf{O}
\end{bmatrix}{\mathbf{\overline{V}}_{i}^{H}} \ \text{are the SVD of} \ {\mathbf{\overline{G}}_{i}}, \ \text{where} \ {\mathbf{\overline{S}}_{i}}\in \CC_{r_i}^{r_i\times r_i}, \ i=1,2,...,n_3.$$
Suppose
$${\mathbf{\overline{A}}_{i}}={\mathbf{\overline{V}}_{i}}\begin{bmatrix}
\mathbf{X}_i & \bigstar\\
\bigstar & \bigstar
\end{bmatrix}{\mathbf{\overline{U}}_{i}^{H}},  \ \text{where} \ {\mathbf{{X}}_{i}} \in \CC^{r_i\times r_i}, i=1,2,...,n_3. $$
By \cite{bbj}, $\mary{\mathbf{\overline{A}}_{i}}{\mathbf{\overline{G}}_{i}}$ exist if and only if $\mathbf{{X}}_{i}$,  $i=1,2,...,n_3$, are nonsingular. In this case,
\begin{align*}
\mary{\mathbf{\overline{A}}_{i}}{\mathbf{\overline{G}}_{i}}={\mathbf{\overline{U}}_{i}}
\begin{bmatrix}
             \mathbf{X}_i^{-1} & \mathbf{O}\\
             \mathbf{O}  &  \mathbf{O}
\end{bmatrix}{\mathbf{\overline{V}}_{i}^{H}}, \ i=1,2,...,n_3.
\end{align*}
Thus, $\mary{\ac}{\gc}$ exists if and only if $\mathbf{X}_i$, $i=1,2,...,n_3$, are nonsingular. Also, we have
\begin{align*}
\DCT(\matt(\mary{\mathcal{A}}{\mathcal{G}})
&=\begin{bmatrix}
     \mary{\mathbf{\overline{A}}_1}{\mathbf{\overline{G}}_1} &  &  \\
      & \ddots &  \\
      &  & \mary{\mathbf{\overline{A}}_{n_3}}{\mathbf{\overline{G}}_{n_3}} \\
\end{bmatrix}\\
&=\begin{bmatrix}
    {\mathbf{\overline{U}}_{1}}\begin{bmatrix}
             \mathbf{X}_1^{-1} & \mathbf{O}\\
             \mathbf{O}  &  \mathbf{O}
\end{bmatrix}{\mathbf{\overline{V}}_{1}^{H}}\\
& \ddots & \\
& & {\mathbf{\overline{U}}_{n_3}}\begin{bmatrix}
             \mathbf{X}_{n_3}^{-1} & \mathbf{O}\\
             \mathbf{O}  &  \mathbf{O}
\end{bmatrix}{\mathbf{\overline{V}}_{n_3}^{H}}\\
\end{bmatrix}\\
&=\begin{bmatrix}
             {\mathbf{\overline{U}}_{1}} &   &   \\
             & \ddots   &   \\
             &   &  {\mathbf{\overline{U}}_{n_3}} \\
\end{bmatrix}\begin{bmatrix}
               \begin{bmatrix}
  \mathbf{X}_1^{-1} & \mathbf{O}   \\
  \mathbf{O} & \mathbf{O}   \\
\end{bmatrix} &   &   \\
                       & \ddots   &   \\
                       &   &  \begin{bmatrix}
  \mathbf{X}_{n_3}^{-1} & \mathbf{O}   \\
  \mathbf{O} & \mathbf{O}   \\
\end{bmatrix} \\
                   \end{bmatrix}\begin{bmatrix}
                     {\mathbf{\overline{V}}_{1}^{H}} &   &   \\
                       & \ddots   &   \\
                       &   &  {\mathbf{\overline{V}}_{n_3}^{H}} \\
                   \end{bmatrix}.
\end{align*}
If $rank\left(L(\scc^{(i)})\right)=r$, $i=1,2,...,n_3$, then one has $rank(\mathbf{X}_1)=rank(\mathbf{X}_2)=\cdots=rank(\mathbf{X}_{n_3})=r$. Implementing $\ten(\IDCT)(\cdot)$ on both sides of the equation above, we have
\begin{equation*} 
\mary{\ac}{\gc} = \uc*_c\mat{\mathcal{X}^{-1}}{\mathcal{O}}{\mathcal{O}}{\mathcal{O}}
*_c\mathcal{V}^H.
\end{equation*}
\end{proof}

\subsection{The algorithm for computing the inverse along a tensor}
In the following, we establish an algorithm to compute the inverse along a tensor by using Theorem \ref{15t}.

\begin{algorithm}[H]
\caption{\textsc{Compute the inverse of $\mathcal{A}$ along a tensor $\mathcal{G}$}}\label{algo9}
\KwIn {$n_1\times n_2\times n_3$ tensor $\mathcal{A}$ and $n_2\times n_1\times n_3$ tensor $\mathcal{G}$}
\KwOut{$n_2\times n_1\times n_3$ tensor $\mathcal{X}$}
\begin{enumerate}
\addtolength{\itemsep}{-0.8\parsep minus 0.8\parsep}

\item $\widehat{\mathcal{A}} = L(\mathcal{A})=\mathcal{A}\times_3\mathbf{M}$,
    $\widehat{\mathcal{G}} = L(\mathcal{G})=\mathcal{G}\times_3\mathbf{M}$,
    where $\mathbf{M}$ is defined in (\ref{l201})

\item  for $i=1,\ldots, n_3$

\qquad  $\svd(\widehat{\mathcal{G}}^{(i)})=\mathbf{U}_i\mathbf{\Sigma}_i\mathbf{V}_i^H$; \\
\qquad  rank$(\mathbf{\Sigma}_i)=r_i$; \\
\qquad  $\mathbf{V}_i^H\widehat{\mathcal{A}}^{(i)}\mathbf{U}_i$
=$\mat{\mathbf{X}_i}{\bigstar}{\bigstar}{\bigstar}$, where $\mathbf{X}_i \in \CC^{r_i\times r_i}$;\\
\qquad  If $\mathbf{X}_i$ is nonsingular,
$\widehat{\mathcal{W}}^{(i)}=\mat{\mathbf{X}_i^{-1}}{\mathbf{O}}{\mathbf{O}}{\mathbf{O}}$;\\
\qquad  $\widehat{\mathcal{Z}}^{(i)}=\mathbf{U}_i\widehat{\mathcal{W}}^{(i)}\mathbf{V}_i^H$;\\
\qquad  $i=i+1$;\\
\qquad  else Output: $\mathcal{A}$ is not invertible along $\mathcal{G}$.\\
end

\item  ${\mathcal{X}} = L^{-1}(\widehat{\mathcal{Z}})=\widehat{\mathcal{Z}}\times_3\mathbf{M}^{-1}$\\
    end
\end{enumerate}
\end{algorithm}

\begin{example}
Let $\ac ,\gc\in \CC^{3\times 3\times 3}$ with frontal slices\\
$$\ac^{(1)}=\begin{bmatrix}
    1  & 0 &0\\
    0  & -1 & 0\\
    3  & 0 & 0\\
  \end{bmatrix},\quad
\ac^{(2)}=\begin{bmatrix}
     0  & 0 &3\\
    5  & 2 & 0\\
    0  & 0 & 1\\
  \end{bmatrix},\quad
  \ac^{(3)}=\begin{bmatrix}
     0  & 2 &0\\
    0  & 0 & 2\\
    0  & 4 & 3\\
  \end{bmatrix},$$
$$\gc^{(1)}=\begin{bmatrix}
    3  & 0 &0\\
    1  & 0 & 0\\
    0  & 0 & 2\\
  \end{bmatrix},\quad
\gc^{(2)}=\begin{bmatrix}
   1  & 0 &5\\
    2  & 0 & 0\\
    2  & 0 & 1\\
  \end{bmatrix},\quad
\gc^{(3)}=\begin{bmatrix}
   0  & 3 &4\\
    1  & 0 & 3\\
    1  & 0 & 0\\
  \end{bmatrix}. $$
By using Algorithm \ref{algo9}, we get
$\mary{\ac} {\gc}\in \CC^{3\times 3\times 3}$ with frontal slices
\begin{eqnarray*}
(\mary{\ac} {\gc})^{(1)}=\begin{bmatrix}
    -0.1043  & -0.0495 &0.1030\\
    0.4039  & -0.1304 & -0.2377\\
    -0.4616  & 0.0521 & 0.1951\\
  \end{bmatrix},\quad
(\mary{\ac} {\gc})^{(2)}=\begin{bmatrix}
    0.1220  & 0.1565 &-0.0864\\
    -0.4423  & 0.1439 & 0.1765\\
    0.5999  & -0.0208 & -0.2729\\
  \end{bmatrix},
\end{eqnarray*}
$$(\mary{\ac} {\gc})^{(3)}=\begin{bmatrix}
    -0.0972  & -0.0769 &0.0281\\
    0.0075  & -0.1129 & 0.1342\\
    -0.1260  & 0.0084 & 0.0486\\
  \end{bmatrix}.$$

\end{example}

\section{Applications to higher-order Markov Chains}

Let $\pc\in \RR^{n\times n\times n}$ be a tensor and
\begin{equation*} 
\DCT(\matt(\pc)) = \begin{bmatrix} L(\pc)^{(1)} & & \\ & \ddots & \\ & & L(\pc)^{(n)} \end{bmatrix}.
\end{equation*}

A higher order Markov chain is an extension of a finite Markov chain, in which the stochastic
process $X_0,X_1, \cdots$ with values in $\{1, 2,\cdots, n\}$, has the transition probabilities
\begin{equation*}
0\leq \pc_{i_1i_2i_3}=\prob(X_t=i_1 \mid X_{t-1}=i_2, X_{t-2}=i_3)\leq1,
\end{equation*}
where  $\sum \limits_{j=1}^{n}L(\pc)^{(i)}_{jk}=1$, $i=1,...,n$, $1\leq k\leq n$. We call the tensor
$\mathcal{P}$ a \textbf{transition tensor}.

Let $F$ be a subset of $\RR$ and let $\{X_t: t\in F\}$ be a set of random variables. If $F$ is
countable and if the range of each $X_t$ is the same finite set, then the chain is said to be
a \textbf{finite Markov chain}. Let us denote $\{G_1, \ldots, G_m \}$ the range of any $X_t$.

It is useful to have in mind that $X_k$ is the outcome of the chain on the $k$th step.
The probability of $X_k$ being in state $G_j$ provided that $X_{k-1}$ was in state $G_i$ is
$L(\pc)^{(i)}_{jk}(s) = \prob(X_s=G_k|X_{s-1}=G_j)$, $i=1,\ldots,n$. These probabilities are said to be the
\textbf{one-step transition probabilities}. If each of the one-step transition probabilities does not
depend on $s$ (does not depend on time), i.e., $L(\pc)^{(i)}_{jk}(s)=\pc_{jki}$, for any $s=1,2,\ldots$,
$i=1,\ldots,n$, then we say that the chain is \textbf{homogeneous}.

In the sequel, we will focus our attention to finite homogeneous Markov chains and will simply write
`Markov chain' or `chain' to denote a finite homogeneous Markov chain.

An \textbf{ergodic set} $\Omega$ is a set of states in which every state of $\Omega$ is accessible
from any other state of $\Omega$ and, in addition, no state outside $\Omega$ is accessible from any state
of $\Omega$.

A \textbf{transient} set $\Omega$ is a set whose elements are states in which every state of $\Omega$
is accessible from every other state of $\Omega$, but some state outside $\Omega$ is accessible from
each state of $\Omega$.

A Markov chain is \textbf{ergodic} if the transition tensor of the chain is irreducible, or equivalently,
the states of this chain form a single ergodic set. An ergodic chain is \textbf{regular} if its transition
tensor $\mathcal{P}$ has the following property: exists a natural number $k$ such that $\pc^k>0$.

A state is \textbf{absorbing} if the chain enters in this state, it can never be left. A chain is an
\textbf{absorbing chain} if it has at least one absorbing state and, in addition, from every state of
this chain it is possible to enter in an absorbing state (but not necessarily in one step).
See \cite{CamMey} for details.

\begin{theorem}\label{yosh}
If $\pc \in \RR^{n\times n\times n}$ is any transition tensor and if $\ac = \ic-\pc$, then $\ac^\#$
exists.
\end{theorem}

\begin{proof}
Let
\begin{equation*} 
\DCT(\matt(\pc)) = \begin{bmatrix} L(\pc)^{(1)} & & \\ & \ddots & \\ & & L(\pc)^{(n)} \end{bmatrix}.
\end{equation*}
Since $\ac = \ic-\pc$, then
\begin{equation*} 
\DCT(\matt(\ac))
=\begin{bmatrix} L(\ac)^{(1)} & & \\ & \ddots & \\ & & L(\ac)^{(n)} \end{bmatrix} =
\begin{bmatrix} L(\mathcal{I})^{(1)}-L(\pc)^{(1)} & & \\ & \ddots & \\ &  & L(\mathcal{I})^{(n)}-L(\pc)^{(n)} \end{bmatrix}.
\end{equation*}
Hence, we have $L(\ac)^{(i)}=L(\mathcal{I})^{(i)}-L(\pc)^{(i)}$ for $i=1,2,\ldots,n$. By using \cite[Theorem 8.2.1]{CamMey}, one has
$\ind(L(\ac)^{(i)})=1$, which implies that the group inverse of $L(\ac)^{(i)}$ exists. Then the group inverse of $\ac$ exists.
\end{proof}

\begin{theorem}
Let $\pc\in\RR^{n\times n\times n}$ be the transition tensor of a chain and let $\ac = \ic-\pc$. Then
\begin{equation*} 
\ic - \ac*_c\ac^\# = \left\{
\renewcommand{\arraystretch}{1.5}  \begin{array}{l}
\lim\limits_{n\rightarrow\infty}\frac{\ic+\pc+\pc^2+\cdots+\pc^{n-1}}{n}, \
\text{for every transition tensor} \ \pc  \\
\lim\limits_{n\rightarrow\infty}(\alpha\ic+(1-\alpha)\pc), \
\text{for every transition tensor} \ \pc \ \text{and} \ 0<\alpha<1  \\
\lim\limits_{n\rightarrow\infty}\pc^n, \ \text{for every regular chain}  \\
\lim\limits_{n\rightarrow\infty}\pc^n, \ \text{for every absorbing chain.}
\end{array} \renewcommand{\arraystretch}{1}
\right.
\end{equation*}
\end{theorem}

\begin{proof}
Let
$$\DCT(\matt(\ac))
=\begin{bmatrix} L(\ac)^{(1)} & & \\ & \ddots & \\ & & L(\ac)^{(n)} \end{bmatrix}.$$
Since $\ac = \ic - \pc$, by Theorem \ref{yosh}, the group inverse of $\ac$ exists. Hence,
\begin{equation*} 
\DCT(\matt(\ac^\#)) =
\begin{bmatrix} (L(\ac)^{(1)})^\# & & \\ & \ddots & \\ & & (L(\ac)^{(n)})^\# \end{bmatrix}.
\end{equation*}
Now, it is easy to see
\begin{equation*} 
\DCT(\matt(\ic-\ac*_c\ac^\#)) =
\begin{bmatrix} L(\ic)^{(1)}-(L(\ac)^{(1)})(L(\ac)^{(1)})^\# & & \\ & \ddots & \\ & & L(\ic)^{(n)}-(L(\ac)^{(n)})(L(\ac)^{(n)})^\# \end{bmatrix}.
\end{equation*}
Denote $L(\ic)^{(i)}=\overline{\mathbf{I}}_i$, $L(\ac)^{(i)}=\overline{\mathbf{A}}_i$ and $L(\pc)^{(i)}=\overline{\mathbf{P}}_i$. By using  \cite[Theorem 8.2.2]{CamMey}, we have
\begin{equation*}
\overline{\mathbf{I}}_i - \overline{\mathbf{A}}_i\overline{\mathbf{A}}_i^\# =
\renewcommand{\arraystretch}{1.5} \left\{
\begin{array}{l}
\lim\limits_{n\rightarrow\infty} \frac{\overline{\mathbf{I}}_i+\overline{\mathbf{P}}_i+\overline{\mathbf{P}}_i^2+
\cdots+\overline{\mathbf{P}}_i^{n-1}}{n}, \
\text{for every transition matrix} \ \overline{\mathbf{P}}_i \\
\lim\limits_{n\rightarrow\infty}(\alpha\overline{\mathbf{I}}_i+(1-\alpha)\overline{\mathbf{P}}_i), \
\text{for every transition matrix} \ \overline{\mathbf{P}}_i \ \text{and} \  0<\alpha<1 \\
\lim\limits_{n\rightarrow\infty}{\overline{\mathbf{P}}_i}^{n}, \ \text{for every regular chain} \\
\lim\limits_{n\rightarrow\infty}{\overline{\mathbf{P}}_i}^{n}, \ \text{for every absorbing chain,} \ i=1,\cdots,n,
\end{array} \renewcommand{\arraystretch}{1}
\right.
\end{equation*}
which implies that
\begin{equation*}
\ic-\ac*_c\ac^\# = \left\{
\renewcommand{\arraystretch}{1.5}  \begin{array}{l}
\lim\limits_{n\rightarrow\infty}\frac{\ic+\pc+\pc^2+\cdots+\pc^{n-1}}{n}, \
\text{for every transition tensor} \ \pc \\
\lim\limits_{n\rightarrow\infty}(\alpha\ic+(1-\alpha)\pc), \ \text{for every transition tensor} \
\pc \ \text{and} \ 0<\alpha<1 \\
\lim\limits_{n\rightarrow\infty}\pc^n, \ \text{for every regular chain} \\
\lim\limits_{n\rightarrow\infty}\pc^n, \ \text{for every absorbing chain.}
\end{array} \renewcommand{\arraystretch}{1}
\right.
\end{equation*}
\end{proof}



~

{\bf\large Funding}

~

This work was supported by the
Guangxi Natural Science Foundation (No.2024GXNSFAA010503) and the National Natural Science Foundation of China (No.12061015).

~

{\bf\large Disclosure statement}

~

No potential conflict of interest was reported by the authors.

~

{\bf\large  Acknowledgements}

~

The authors would like to thank the editor and the anonymous reviewers for their valuable comments, which
have significantly improved the paper.

\end{spacing}

\begin{thebibliography}{99}

\bibitem{K} P. Kroonenberg, Three-Mode Principal Component Analysis: Theory and Applications, DSWO Press, Leiden,
1983.





\bibitem{NCT} M. Ng, R. Chan, W. Tang, A fast algorithm for deblurring models with Neumann boundary
conditions, SIAM J. Sci. Comput. (1999) pp. 851-866.


\bibitem{HKBH}
N. Hao, M. Kilmer, K. Braman, R. Hoover, Facial recognition using tensor-tensor decompositions,
Siam J Imaging Sci, 6 (2013) pp. 437-463.


\bibitem{Comon}
P. Comon, Tensor decompositions, in Mathematics in Signal Processing V, J. G. McWhirter
and I. K. Proudler, eds., Clarendon Press, Oxford, UK, 2002, pp. 1-24.

\bibitem{Lathauwer}
L. De Lathauwer, B. De Moor, From matrix to tensor: Multilinear algebra and signal
processing, in Mathematics in Signal Processing IV, J. McWhirter and I. K. Proudler, eds.,
Clarendon Press, Oxford, UK, 1998, pp. 1-15.

\bibitem{Nagy}
J. Nagy, M. Kilmer, Kronecker product approximation for preconditioning in threedimensional
imaging applications, IEEE Trans. Image Process., 15 (2006), pp. 604-613.

\bibitem{Sidiropoulos}
N. Sidiropoulos, R. Bro, G. Giannakis, Parallel factor analysis in sensor array processing,
IEEE Trans. Signal Process., 48 (2000), pp. 2377-2388.

\bibitem{Hoge}
W. Hoge, C. Westin, Identification of translational displacements between Ndimensional data
sets using the high order SVD and phase correlation, IEEE Trans. Image Process., 14 (2005), pp. 884-889.

\bibitem{Rezghi}
M. Rezghi, L. Eld\'{e}n, Diagonalization of tensors with circulant structure, Linear Algebra
Appl., 435 (2011), pp. 422-447.


\bibitem{Sun}
L. Sun, B. Zheng, C. Bu, Y. Wei, Moore-Penrose inverse of tensors via Einstein product, Linear
Multilinear Algebra, 64(4) (2016), pp. 686-698.

\bibitem{Sun2}
L. Sun , B. Zheng , Y. Wei, C. Bu, Generalized inverses of tensors via a general product of tensors,
Front. Math. China, 13(4)  (2018), pp. 893-911.

\bibitem{Miao}
Y. Miao, L. Qi, Y. Wei, Generalized tensor function via the tensor singular value decomposition based
on the T-product. Linear Algebra Appl., 590 (2020), pp. 258-303.


\bibitem{MQW} Y. Miao, L. Qi, Y. Wei, T-Jordan Canonical Form and T-Drazin Inverse based on the
T-Product, Communications on Applied Mathematics and Computation, 3 (2021), pp. 201-220.


\bibitem{PBM}
K. Panigrahy, R. Behera, D. Mishra,
Reverse order law for the Moore-Penrose inverses of tensors,
Linear Multilinear Algebra, 68 (2020), pp. 246-264.

\bibitem{BSMN} R. Behera, J. Sahoo, R. Mohaptra, M. Nashed, Computation of
Generalized Inverses of Tensors via t-Product, Numerical Linear Algebra with
Applications, 29 (2021). e2416.

\bibitem{Liu}
Y. Liu, H. Ma, Dual core generalized inverse of third-order dual tensor based on the T-product.
Comput. Appl. Math., 41(8) (2022), 391(1-28).


\bibitem{Cong2}
Z. Cong, H. Ma,
Characterizations and perturbations of the core-EP inverse of tensors based on the T-product.
Numer. Funct. Anal. Optim., 43 (10) (2022), pp. 1150-1200.

\bibitem{Sahoo}
J. K. Sahoo, R. Behera, P. S. Stanimirovi\'c, V. N. Katsikis, H. Ma, Core and core-EP inverses of tensors. Comput. Appl. Math., 39 (1) (2020), https://doi.org/10.1007/s40314-019-0983-5.



\bibitem{Jin}
H. Jin, M. Bai, J. Ben\'itez, X. Liu, The generalized inverses of tensors and an application to linear models, Comput Math Appl, 74 (2017), pp. 385-397.

\bibitem{Behera}
R. Behera, D. Mishra, Further results on generalized inverses of tensors via the Einstein product. Linear Multilinear Algebra 65(8) (2017), pp. 1662-1682.



\bibitem{JW} J. Ji, Y. Wei, The Drazin inverse of an even-order tensor and its application
to singular tensor equations, Comput. Math. Appl., 75(9), (2018), pp. 3402-3413.

\bibitem{Behera2}
R. Behera, A. Nandi, J. Sahoo, Further results on the Drazin inverse of even order tensors, Numer Linear Algebr, 27 (5) (2020), e2317.







\bibitem{bbj} J. Ben\'itez, E. Boasso, H. Jin. On one-sided (B, C)-inverses of arbitrary matrices, Electron J Linear Al, 32 (2017), pp. 391-422.


\bibitem{KB} T. Kolda, B.  Bader, Tensor decompositions and applications,
SIAM review, 51 (2009), pp. 455-500.



\bibitem{KKA} E. Kernfeld, M. Kilmer, S. Aeron, Tensor-tensor products
with invertible linear transforms, Linear Algebra Appl., 485
(2015), pp. 545-570.


\bibitem{XZN} W. Xu, X. Zhao, M. Ng, A fast algorithm for cosine transform based tensor singular value
decomposition, (2019), arXiv:1902.03070.



\bibitem{BEJR} A. Bentbib, A. El Hachimi, K. Jbilou,  A. Ratnani, Fast multidimensional completion and principal component analysis methods via the cosine product, Calcolo, 59(3) (2022),
https://doi.org/10.1007/s10092-022-00469-2.



















\bibitem{SPKS} P. Stanimirovi\'{c}, D. Pappas, V. Katsikis, I. Stanimirovi\'{c}, Symbolic computation of $A^{(2)}_{T,S}$-inverses using QDR factorization, Linear Algebra
Appl., 437 (2012)  pp. 1317-1331.

\bibitem{BT} O. Baksalary, G. Trenkler, Core inverse of matrices, Linear Multilinear Algebra, 58 (2010) 681-697.



\bibitem{WWQ}
G. Wang, Y. Wei, S. Qiao.
Generalized inverses: Theory and Computations.
Developments in Mathematics 53.
Singapore, Springer, Beijing, 2018. Science Press.

\bibitem{SD} P. Stanimirovi\'{c}, D. Djordjevi\'{c}, Full-rank and determinantal representation of the Drazin inverse, Linear Algebra Appl., 311 (2000) pp. 131-151.


\bibitem{C} R. Cline, Inverses of rank invariant powers of a matrix, SIAM J. Numer. Anal. 5 (1) (1968) pp. 182-197.

\bibitem{MT} S. Malik, N. Thome, On a new generalized inverse for matrices of an arbitrary index, Appl Math Comput, 226 (2014) 575-580.

\bibitem{mary2} X. Mary, P. Patr\'{\i}cio. The inverse along a lower triangular matrix,
Applied Mathematics and Computation 219 (2012), pp. 886-891.




\bibitem{CamMey}
S.L. Campbell, C.D. Meyer, Generalized Inverse of Linear Transformations, Pitman, London, 1979;
Dover, New York, 1991.



















\end{thebibliography}
\end{document}